\documentclass[a4paper]{article}
\usepackage{geometry,authblk}
\usepackage{graphicx,url}
\usepackage[margin=10pt, font=small, labelfont=bf]{caption}

\usepackage{slashed}
\usepackage{mathtools}
\mathtoolsset{showonlyrefs}

\usepackage[english]{babel}
\usepackage[T1]{fontenc}
\usepackage[utf8]{inputenc}

\usepackage{tikz}
\usepackage{verbatim}
\usepackage{listings}
\usepackage{xcolor}
\usepackage{amsmath,amssymb,mathrsfs}
\usepackage{mathtools}
\usepackage{amsthm}
\usepackage{frontespizio}
\usepackage{pictexwd, dcpic}
\usepackage{epigraph}
\usepackage{enumerate}
\usepackage{fancyhdr}
\usepackage{textcomp}
\usepackage{cancel}
\usepackage{bbm}
\theoremstyle{plain}
\newtheorem{thm}{Theorem}[section]

\newtheorem{prop}[thm]{Proposition}
\newtheorem{cor}[thm]{Corollary}
\newtheorem{lem}[thm]{Lemma}
\theoremstyle{definition}
\newtheorem{defn}[thm]{Definition}
\newtheorem{rem} [thm] {Remark}

\title{Local and global solutions on arcs for the Ericksen - Leslie problem in $\mathbb{R}^N$}
\date{}
\author{Daniele Barbera}
\author[1,2,3]{Vladimir Georgiev}

\affil[1]{Dipartimento di Matematica  \\ Universit\`a di Pisa \\ Largo B. Pontecorvo 5, 56100 Pisa, Italy}
\affil[2]{Faculty of Science and Engineering \\ Waseda University \\
 3-4-1, Okubo, Shinjuku-ku, Tokyo 169-8555 \\
Japan}
\affil[3]{Institute of Mathematics and Informatics, Bulgarian Academy of Sciences, Acad.
Georgi Bonchev Str., Block 8, 1113 Sofia, Bulgari}
\begin{document}
\maketitle
\thanks{
  D.B. and V.G. were partially supported by INDAM, GNAMPA group. V.G. is supported by the project PRIN  2020XB3EFL by the Italian Ministry of Universities and Research,  by   the Top Global University Project, Waseda University, by the University of Pisa, Project PRA 2018 49 and by Institute of Mathematics and Informatics, Bulgarian Academy of Sciences.
}

\begin{abstract}
    The work deals with the Ericksen-Leslie System for nematic liquid crystals on the space $\mathbb{R}^N$ with $N\ge 3$. In our work we suppose the initial condition $v_0$ stays on an arc connecting two fixed orthogonal vectors on the unit sphere. Thanks to this geometric assumption, we  prove through  energy a priori estimates the local existence and the global existence for small initial data of a solution
$$ u\in L^\infty((0,T);H^s(\mathbb{R}^N)),\quad \nabla u\in L^2((0,T);H^s(\mathbb{R}^N)), $$
$$ \nabla v\in L^\infty((0,T);H^s(\mathbb{R}^N)),\quad \nabla^2v\in L^2((0,T);H^s(\mathbb{R}^N)) $$
for $s>\frac{N}{2}-1$, asking low regularity assumptions  on $u_0$ and $v_0.$

\end{abstract}

\section{Introduction}

The aim of the paper is to study the  existence and dispersive properties of solutions to the  Ericksen-Leslie model:
\begin{equation}\label{EL.sys.}
    \left\{\begin{array}{ll}
       (\partial_t-\Delta)u+\nabla p+u\cdot \nabla u=-{\rm Div}\left(\nabla v\odot\nabla v\right)  & \mathbb{R}_+\times \mathbb{R}^N \\
       {\rm div}u=0 & \mathbb{R}_+\times\mathbb{R}^N \\
       (\partial_t-\Delta)v +u\cdot \nabla v = |\nabla v|^2v & \mathbb{R}_+\times\mathbb{R}^N \\
       |v|=1 & \mathbb{R}_+\times\mathbb{R}^N,
    \end{array}\right.
\end{equation}
with

$$ (\nabla w \odot \nabla z)_{j,k}=\partial_jw\cdot \partial_kz \quad \forall j,k=1,\ldots, N,\quad \forall w,z\colon\mathbb{R}^N\to\mathbb{R}^N, $$
$$ Div A=\sum_{j=1}^N\partial_j A^j\quad \forall A\colon\mathbb{R}^N\to\mathbb{R}^{N^2}. $$
The Ericksen-Leslie theory was introduced by Ericksen in \cite{E61} and later by Leslie in \cite{L79} in order to study the behaviour of nematic liquid crystals (for a more physical analysis see \cite{FP22}). The system \eqref{EL.sys.} was implemented by Lin in \cite{L89} as a simplified version of the model derived by Ericksen and Leslie. Here $u\colon\mathbb{R}_+\times \mathbb{R}^N\to\mathbb{R}^N$ is the velocity field of the fluid, $p\colon\mathbb{R}_+\times\mathbb{R}^N\to\mathbb{R}$ is the pressure term which arises from the divergence-free condition over $u$ and $v\colon\mathbb{R}_+\times\mathbb{R}^N\to \mathcal{S}^{N-1}$ is the orientation field of the liquid crystals particles. The system \eqref{EL.sys.} can be seen as a nonlinear coupling of the Navier-Stokes equation for incompressible fluids with the flow of harmonic maps in $S^{N-1}$. Precisely, the first equation of \eqref{EL.sys.} is the conservation of the linear momentum, while the third one is the conservation of the angular momentum (for a more detailed explanation see the appendix of \cite{L89}). On the other hand, the Ericksen-Leslie model is not the only model known for the study of liquid crystals: one of the most famous if the Beris-Edward model, also known as the $Q$-tensor model. As a classical reference, it can be seen \cite{PZ11} and \cite{PZ12}.

In \cite{LL96} and \cite{LL00} Lin-Liu, in order to avoid the condition $|v|=1$, used the Ginzburg-Landau approximation replacing the nonlinearity $|\nabla v|^2v$ with $\frac{1}{\varepsilon^2}(1-|v|^2)v$ with $\varepsilon>0$. They proved the existence of a global solution and a regularity result.  More recent results, based on  Ginzburg-Landau approximation, can be found in \cite{FT09}, \cite{DQ12} and \cite{DH22}.

For what concerns the system \eqref{EL.sys.} without approximations, we can cite the paper \cite{W11}. Here, Wang proved local and global smalldata existence of unique solution to the system \eqref{EL.sys.} for
$$ u_0\in \text{BMO}^{-1}(\mathbb{R}^N;\mathbb{R}^N), \quad d_0\in \text{BMO}(\mathbb{R}^N;\mathcal{S}^2). $$
From his result, it follows the local and global existence for $u_0\in L^N(\mathbb{R}^N;\mathbb{R}^N)$ and $v_0\in W^{1,N}(\mathbb{R}^N;\mathcal{S}^2)$ with
$$ \|u\|_{L^\infty(\mathbb{R}^N)}+\|\nabla v\|_{L^\infty(\mathbb{R}^N)}\lesssim t^{-\frac{1}{2}}\left(\|u_0\|_{L^N(\mathbb{R}^N)}+\|d_0\|_{W^{1,N}(\mathbb{R}^N)}\right). $$

The use of $BMO$ type spaces essentially uses the explicit representation of the heat kernel in $\mathbb{R}^n$. Our goal is to follow another idea that gives also approach for the case of  boundary value problem.

Finally, we recall the work of Liu-Xu in \cite{LX15} and later  the one of Liu-Wu-Xu-Zhang in \cite{XZ17}. In the first paper, Liu and Xu managed to prove global existence for \eqref{EL.sys.} with $u_0\in H^m(\mathbb{R}^3;\mathbb{R}^3)$ and $v_0\in H^{m+1}(\mathbb{R}^3;\mathcal{S}^2)$ for $m\ge 3$ and
$$ \|u_0\|_{L^2(\mathbb{R}^3)}+\|\nabla v_0\|_{L^2(\mathbb{R}^3)}\le \varepsilon\ll1. $$
Moreover, if $u_0,d_0\in L^1(\mathbb{R}^3;\mathbb{R}^3)$, then
$$ \|\nabla^lu\|_{L^2(\mathbb{R}^3)}+\|\nabla^{l+1}v\|_{L^2(\mathbb{R}^3)}\lesssim t^{-\frac{3}{4}-\frac{\ell}{2}}\quad \text{as}\:\:t\to+\infty\quad  \forall \ell=0,\ldots, m. $$
In the second paper, the authors proved the same (also for the compressible case) for $u_0\in H^2(\mathbb{R}^3;\mathbb{R}^3)$, $v_0\in H^3(\mathbb{R}^3;\mathcal{S}^2)$ and
$$ \|u_0\|_{H^2(\mathbb{R}^3)}+\|\nabla v_0\|_{H^2(\mathbb{R}^3)}\le \varepsilon\ll1. $$
Here again, if $u_0,v_0\in L^1(\mathbb{R}^3;\mathbb{R}^3)$, they get a decay result:
$$ \|(u,v)\|_{L^p(\mathbb{R}^3)}\lesssim t^{-\frac{3}{2}\left(1-\frac{1}{p}\right)}\quad \text{as}\:\:t\to+\infty\quad \forall p\in[2,6]; $$
$$ \|\nabla u\|_{H^1(\mathbb{R}^3)}+\|\nabla v\|_{H^2(\mathbb{R}^3)}\lesssim t^{-\frac{5}{4}}\quad \text{as}\:\:t\to+\infty. $$

The aim of the paper is to prove global existence and time decay working with small low regularity initial data in classical Sobolev spaces $H^s(\mathbb{R}^N)$ with $N \geq 3.$  For this reason we will follow another approach based on the following almost orthogonality invariance property, namely if we consider the scalar product of $v$ of \eqref{EL.sys.} and a fixed vector $\omega \in \mathcal{S}^{N-1}$
$$ \varphi_\omega(t,x) \coloneqq \left<v(t,x), \omega \right> $$
and observe that it satisfies
\begin{equation}\label{s.perp.}
    \left\{\begin{array}{ll}
    (\partial_t-\Delta)\varphi_\omega +u\cdot \nabla \varphi_\omega = |\nabla v|^2\varphi_\omega  & \mathbb{R}_+\times\mathbb{R}^N \\
    \varphi_\omega(0)=\left<v_0,\omega\right> & \mathbb{R}^N.
\end{array}\right.
\end{equation}
 Our idea comes from the following guess: if
the initial data are sufficiently regular and small in Sobolev spaces
and if $$ \left<v_0,\omega\right>=0,$$
then this orthogonality condition is preserved over the nonlinear flow. For this we can take to fixed vectors $$\eta\in \mathcal{S}^{N-1}, \ \omega\in \mathcal{S}^{N-1}$$ with $\omega\perp \eta$ and consider the arc connecting them on the unit sphere.
 We assume the initial data for $v$ are on this arc, i.e.
$$ v_0(x)= \cos d_0(x)\eta + \sin d_0(x)\omega \quad \text{for a.e.}\:\:x\in\mathbb{R}^N, $$
then we expect that the $v$ remains on the same arc, i.e. we have
\begin{equation}\label{ass.v}
    v(t,x)=\cos d(t,x)\eta + \sin d(t,x)\omega  \quad \text{for a.e.}\:\: (t,x)\in\mathbb{R}_+\times\mathbb{R}^N.
\end{equation}
In particular, we have the relations
\begin{equation}
    \begin{aligned}
   &  \partial_t v = -\partial_t d \sin(d) \eta   +  \partial_t d \cos(d) \omega , \\
    & \nabla_x v = -\nabla_x d \sin(d) \eta   + \nabla_x d \cos(d) \omega,  \\
     & |\nabla_x v|^2 =|\nabla_x d|^2, \\
     & \Delta v = -\Delta d \sin(d) \eta - |\nabla_x d|^2 \cos (d)\eta    + \Delta d \cos(d) \omega
     - |\nabla_x d|^2 \sin(d) \omega , \\
    \end{aligned}
\end{equation}
So, if we substitute the ansatz \eqref{ass.v} into \eqref{EL.sys.} and assume $d$ is sufficiently regular (for example  $ d \in  L^\infty, \nabla d \in H^s, s >N/2-1,$) we use the equation
$$ (\partial_t-\Delta)v +u\cdot \nabla v = |\nabla v|^2v $$
and obtain  two scalar  equations on the vectors $\eta$ and $\omega.$
\begin{equation}
    \left\{
    \begin{aligned}
   & -\partial_t d \sin(d) +\Delta d \sin(d) + |\nabla d|^2 \cos (d) - u\cdot \nabla_x d \sin(d) =  |\nabla d|^2 \cos(d), \\
   & \partial_t d \cos(d) - \Delta d \cos(d) + |\nabla d|^2 \sin(d) +  u\cdot \nabla_x d \cos(d) = |\nabla d|^2  \sin(d),
    \end{aligned}\right.
\end{equation}
which can be rewrited as
\begin{equation}
\left\{    \begin{aligned}
   & \sin(d)\left(-\partial_t d +\Delta d - u\cdot \nabla_x d \right) =0, \\
   & \cos(d)\left(\partial_t d  - \Delta d  +  u\cdot \nabla_x d\right) =0 .
    \end{aligned} \right.
\end{equation}
Hence, we arrive at the equation
\begin{equation}
    \partial_t d - \Delta d   +  u\cdot \nabla_x d  =0 .
\end{equation}
For what concerns the first equation of (\ref{EL.sys.}), we notice that
$$ (\nabla v\odot \nabla v)_{j,k}=\partial_jd\partial_k d=(\nabla d\otimes \nabla d)_{j,k} \quad \forall j,k=1,\ldots, N, $$
where
$$ (w\otimes z)_{j,k}\coloneqq w_jz_k\quad \forall w,z\colon\mathbb{R}^N\to\mathbb{R}^N. $$
Finally, (\ref{EL.sys.}) becomes
\begin{equation}\label{EL.red.sys.}
    \left\{\begin{array}{ll}
        (\partial_t-\Delta)u+\nabla p+u\cdot \nabla u=-{\rm Div}(\nabla d\otimes\nabla d) & \mathbb{R}_+\times \mathbb{R}^N \\
        {\rm div}u=0 & \mathbb{R}_+\times\mathbb{R}^N \\
        (\partial_t-\Delta)d+u\cdot \nabla d=0 & \mathbb{R}_+\times \mathbb{R}^N \\
        u(0)=u_0,\quad d(0)=d_0 & \mathbb{R}^N.
    \end{array}\right.
\end{equation}
Thanks to the previous considerations, in the first part of the paper we will focus on \eqref{EL.red.sys.}. In particular, we will study local and global existence for the reduced system \eqref{EL.red.sys.} with decay estimates as $t\to+\infty$. Let us recall the following spaces: let $s\ge 0$ and $q\in(1,\infty)$, then it is well-known (see for example \cite{HRS20}) the space
$$ H^s_q\left(\mathbb{R}^N\right)\coloneqq \overline{C^\infty_c(\mathbb{R}^N)}^{\|\cdot\|_{H^s_q}}, $$
where
$$ \|v\|_{H^s_q(\mathbb{R }^N)}=\|v\|_{H^{\lfloor s\rfloor}_q(\mathbb{R}^N)}+\sum_{|\alpha|=\lfloor s\rfloor}\left\|\left\|\frac{D^\alpha v(x+h)- D^\alpha v(x)}{|h|^{s-\lfloor s\rfloor}}\right\|_{L^2\left(\frac{dh}{|h|^N}\right)}\right\|_{L^q(\mathbb{R}^N)}. $$
Following \cite{G11} it is defined
$$ D^{1,2}\left(\mathbb{R}^N\right)\coloneqq \left\{w\in L^1_{loc}\left(\mathbb{R}^N\right)\:\big|\:\|\nabla w\|_{L^2(\mathbb{R}^N)}<+\infty\right\} $$
and
$$ \widehat{H}^{1}\left(\mathbb{R}^N\right)=\left\{[w]\:\big|\: w\in D^{1,2}\left(\mathbb{R}^N\right)\right\} $$
where $[w]$ stands for the equivalence class defined as follows:
$$ [w]\coloneqq\{z=w+c \mid c\in\mathbb{R}\}. $$
In \cite{OS12} it is proved that $\widehat{H}^1(\mathbb{R}^N)$ is a Banach space endowed with the norm
$$ \|[w]\|_{\widehat{H}^1(\mathbb{R}^N)}\coloneqq \|\nabla w\|_{L^2(\mathbb{R}^N)}, $$
and that
\begin{rem}\label{rem.Hom.Sob.}
In \cite{OS12} they prove that the space
$$ \widehat{\mathcal{D}}\left(\mathbb{R}^N\right)\coloneqq\left\{[\varphi]\:\big|\: \varphi\in C^\infty_c\left(\mathbb{R}^N\right)\right\} $$
is dense in $\widehat{H}^1(\mathbb{R}^N)$.
For this reason, we can prove that for any $[w]\in\widehat{H}^1(\mathbb{R}^N)$ there is a representative function $u\in [w]$ such that
$$ \|u\|_{L^\frac{2N}{N-2}(\mathbb{R}^N)} \lesssim \|\nabla w\|_{L^2(\mathbb{R}^N)}. $$
In fact, if $w_n$ is a sequence of test functions such that $[w_n]\to[w]$ in $\widehat{H}^1(\mathbb{R}^N)$, then
$$ \|w_n-w_m\|_{L^\frac{2N}{N-2}(\mathbb{R}^N)}\lesssim \|\nabla w_n-\nabla w_m\|_{L^2(\mathbb{R}^N)}\xrightarrow{n,m\to+\infty}0. $$
So we can find $w_\infty\in L^\frac{2N}{N-2}(\mathbb{R}^N)$ such that $w_n\to w_\infty$ in $L^\frac{2N}{N-2}(\mathbb{R}^N)$. On the other hand it is easy to see that $\nabla w_\infty=\nabla w$, so $w_\infty=w+c\in[w]$, which concludes the proof.
\end{rem}
It is also defined (see for example \cite{HRS20})
$$ J_2\left(\mathbb{R}^N\right)\coloneqq \left\{w\in L^2\left(\mathbb{R}^N;\mathbb{R}^N\right)\:\Big|\: \left<w,\nabla \varphi\right>=0\quad \forall [\varphi]\in \widehat{H}^{1}\left(\mathbb{R}^N\right)\right\}.
$$
Finally, we introduce the spaces we will use in  this paper:
\begin{defn}\hfill\\
 Let $s\ge 0$ and $T\in(0,\infty]$, then we define
 \begin{equation}\label{def.X}
    X^s_T\coloneqq \left\{v\in L^2_{loc}\left((0,T);H^{s+1}\left(\mathbb{R}^N;\mathbb{R}^N\right)\right)\:\Big|\:\|v\|_{X^s_T}<+\infty\right\}
 \end{equation}
 \begin{equation}\label{def.T}
   \Theta^s_T\coloneqq \left\{\theta\in L^\infty\left((0,T);L^\infty\left(\mathbb{R}^N\right)\cap \widehat{H}^1\left(\mathbb{R}^N\right)\right)\:\Big|\: \nabla \theta\in X^s_T\right\},
 \end{equation}
 where for every $t\in(0,T)$
 $$ \theta(t)\in L^\infty\left(\mathbb{R}^N\right)\cap\widehat{H}^1\left(\mathbb{R}^N\right)\:\Leftrightarrow\: \theta(t)\in L^\infty\left(\mathbb{R}^N\right),\quad [\theta(t)]\in \widehat{H}^1\left(\mathbb{R}^N\right) $$
 and
 $$ \|v\|_{X^s_T}\coloneqq \|\nabla v\|_{L^2((0,T);H^s(\mathbb{R}^N))}+\|v\|_{L^\infty((0,T);H^s(\mathbb{R}^N))}, $$
 $$ \|\theta\|_{\Theta^s_T}\coloneqq \|\theta\|_{L^\infty((0,T);L^\infty(\mathbb{R}^N))}+\|\nabla \theta\|_{X^s_T}. $$
 For $T=\infty$ we will write $X^s$ and $\Theta^s$.
\end{defn}

\begin{thm}\label{t.loc.ex.}\hfill\\
Let $N\ge 3$, $s>\frac{N}{2}-1$. For any $R>0$ there exists $T=T(R)>0$ so that for any initial data
$$ u_0\in J_2\left(\mathbb{R}^N\right)\cap H^s\left(\mathbb{R}^N;\mathbb{R}^N\right), \quad d_0\in L^\infty\left(\mathbb{R}^N\right),\quad \nabla d_0\in H^{s}\left(\mathbb{R}^N;\mathbb{R}^N\right) $$
and
$$ \|u_0\|_{H^s(\mathbb{R}^N)}+\|d_0\|_{L^\infty(\mathbb{R}^N)}+\|\nabla d_0\|_{H^{s}(\mathbb{R}^N)} \leq R$$
there exists a solution $(u,p,d)$, unique up to additive functions $\rho(t)$ on the pressure term $p$, for the reduced system \eqref{EL.red.sys.} in $(0,T)$ with
$$ u\in X^s_T,\quad d\in\Theta^s_T,\quad \nabla p\in L^2\left((0,T);H^s\left(\mathbb{R}^N;\mathbb{R}^N\right)\right),\quad p(t)\in\widehat{H}^1\left(\mathbb{R}^N\right)\quad \text{for a.e.}\:\:t\in(0,T),$$
where $X^{s}_T$ and $\Theta^s_T$ are defined in \eqref{def.X} and \eqref{def.T}, and
$$ \|u\|_{X^s_T}+\|\nabla p\|_{L^2((0,T);H^s(\mathbb{R}^N))}+\|d\|_{\Theta^s_T}\le C(N,s)\left[\|u_0\|_{H^s(\mathbb{R}^N)}+\|d_0\|_{L^\infty(\mathbb{R}^N)}+\|\nabla d_0\|_{H^{s}(\mathbb{R}^N)}\right]. $$
\end{thm}
Later in the paper, we will prove the global existence for small initial date and the corresponding decay in time:
\begin{thm}\label{t.gl-dec.ex.}
  Let $N\ge 3$ and $s>\frac{N}{2}-1$. Then there exists $\varepsilon_0>0$ so that for any $\varepsilon \in (0,\varepsilon_0]$ and any initial data
$$ u_0\in J_2\left(\mathbb{R}^N\right)\cap  H^s\left(\mathbb{R}^N;\mathbb{R}^N\right), \quad  d_0\in L^\infty\left(\mathbb{R}^N\right), \quad \nabla d_0\in H^s\left(\mathbb{R}^N;\mathbb{R}^N\right), $$
such that
$$ \|u_0\|_{H^s(\mathbb{R}^N)}+\|d_0\|_{L^\infty(\mathbb{R}^N)}+\|\nabla d_0\|_{H^{s}(\mathbb{R}^N)}\le \varepsilon, $$
 there is a  solution $(u,p,d)$, unique up to additive functions $\rho(t)$ on the pressure term $p$, for the reduced system \eqref{EL.red.sys.} in $\mathbb{R}_+$ with
 $$ u\in X^s,\quad d\in\Theta^s,\quad \nabla p\in L^2\left(\mathbb{R}_+;H^s\left(\mathbb{R}^N;\mathbb{R}^N\right)\right),\quad p(t)\in\widehat{H}^1\left(\mathbb{R}^N\right)\quad \text{for a.e.}\:\:t>0, $$
where $X^{s}$ and $\Theta^s$ are defined in \eqref{def.X} and \eqref{def.T}, and
$$ \|u\|_{X^s}+\|\nabla p\|_{L^2(\mathbb{R}_+;H^s(\mathbb{R}^N))}+\|d\|_{\Theta^s}\le C(N,s)\varepsilon. $$
\end{thm}

One can pose the question how $d(t,x)$ behaves as $t \to \infty$? We have the following answer.

\begin{cor}
Let us suppose the assumptions of Theorem \ref{t.loc.ex.} are fulfilled and $(u,p,d)$ is the solution from the same theorem.
For any $k\in\mathbb{N}$ such that $s-k>\frac{N}{2}-1$, we have
    $$ \|u(t)\|_{W^{k,\infty}(\mathbb{R}^N)}+\|\nabla d(t)\|_{W^{k,\infty}(\mathbb{R}^N)}\le C(N,s) t^{-N/4}\quad \text{as}\:\:t\to+\infty. $$
\end{cor}

Finally, we turn back to the Ericksen-Leslie system \eqref{EL.sys.}. Here we divide the results in two cases: when $s$ is an integer and when $s$ is fractional. This last case is harder, so we restricted ourselves just to the case $N=3$ and $s\in\left(\frac{1}{2},1\right)$.  We will use the notation
$$ \text{Span}\{\eta,\omega\}\coloneqq\left\{z=a_1\eta +a_2\omega\mid a_1,a_2\in\mathbb{R}\right\}\quad \forall \eta,\omega\in \mathbb{R}^N. $$
\begin{thm}\label{t.ex.int.}\hfill\\
Let $N\ge 3$,  $s\in\mathbb{N}$ with $s>\frac{N}{2}-1$, let
$$ u_0\in J_2\left(\mathbb{R}^N\right)\cap H^s\left(\mathbb{R}^N;\mathbb{R}^N\right),\quad v_0\colon \mathbb{R}^N\to \mathcal{S}^{N-1},\quad \nabla v_0\in H^s\left(\mathbb{R}^N;\mathbb{R}^{N^2}\right) $$
with $v_0\in \text{Span}\{\eta,\omega\}$, for some $\eta,\omega\in \mathcal{S}^{N-1}$ with $\eta\perp\omega$, then
\begin{itemize}
    \item For any $R>0$ such that
    $$ \|u_0\|_{H^s(\mathbb{R}^N)}+\|v_0-\eta\|_{L^\infty(\mathbb{R}^N)}+\|\nabla v_0\|_{H^s(\mathbb{R}^N)}\le R, $$
    there is $T=T(R)>0$ sufficiently small such that there exists a solution $(u,p,v)$, unique up to additive functions $\rho(t)$ on the pressure term $p$, for the Ericksen-Leslie system \eqref{EL.sys.} in $(0,T)$ with
    $$ u\in X^s_T,\quad v\in\Theta^s_T,\quad \nabla p\in L^2\left((0,T);H^s\left(\mathbb{R}^N;\mathbb{R}^N\right)\right),\quad p(t)\in\widehat{H}^1\left(\mathbb{R}^N\right)\quad \text{a.e.}\:\:t\in(0,T),$$
    where $X^s_T$ and $\Theta^s_T$ are defined in \eqref{def.X} and \eqref{def.T} and
    $$ \|u\|_{X^s_T}+\|\nabla p\|_{L^2((0,T);H^s(\mathbb{R}^N))}+\|\nabla v\|_{X^s_T}\le $$
    $$ \le C(N,s)\left[\|u_0\|_{H^s(\mathbb{R}^N)}+\|v_0-\eta\|_{L^\infty(\mathbb{R}^N)}+\|\nabla v_0\|_{H^{s}(\mathbb{R}^N)}\right]. $$
    \item There is $\varepsilon_0=\varepsilon_0(s,N)>0$ such that for any $\varepsilon\in[0,\varepsilon_0)$, if
    $$ \|u_0\|_{H^s(\mathbb{R}^N)}+\|v_0-\eta\|_{L^\infty(\mathbb{R}^N)}+\|\nabla v_0\|_{H^s(\mathbb{R}^N)}\le \varepsilon, $$
    then there is a  solution $(u,p,v)$, unique up to additive functions $\rho(t)$ on the pressure term $p$, for the Ericksen-Leslie system \eqref{EL.sys.} in $\mathbb{R}_+$ with
    $$ u\in X^s,\quad v\in\Theta^s,\quad \nabla p\in L^2\left(\mathbb{R}_+;H^s\left(\mathbb{R}^N;\mathbb{R}^N\right)\right),\quad p(t)\in\widehat{H}^1\left(\mathbb{R}^N\right)\quad \text{for a.e.}\:\:t>0, $$ and
    $$ \|u\|_{X^s}+\|\nabla p\|_{L^2(\mathbb{R}_+;H^s(\mathbb{R}^N))}+\|v-\eta\|_{\Theta^s}\le C(N,s)\varepsilon. $$ Moreover, for any $k\in\mathbb{N}$ such that $s-k>\frac{N}{2}-1$, it holds
    $$ \|u(t)\|_{W^{k,\infty}(\mathbb{R}^N)}+\|\nabla v(t)\|_{W^{k,\infty}(\mathbb{R}^N)}\le C(N,s,k)t^{-N/4}\quad \text{as}\:\:t\to+\infty. $$
\end{itemize}
\end{thm}
\begin{thm}\label{t.ex.fr.}\hfill\\
Let $s\in\left(\frac{1}{2},1\right)$ and
$$ u_0\in J_2\left(\mathbb{R}^3\right)\cap H^s\left(\mathbb{R}^3;\mathbb{R}^3\right),\quad v_0\colon \mathbb{R}^3\to \mathcal{S}^2,\quad \nabla v_0\in H^s\left(\mathbb{R}^3;\mathbb{R}^{9}\right) $$
with $v_0\in\text{Span}\{\eta,\omega\}$ for $\eta,\omega\in \mathcal{S}^2$ such that $\eta\perp\omega$, then there is $\varepsilon_0=\varepsilon_0(s)>0$ such that for any $\varepsilon\in[0,\varepsilon_0)$, if
$$ \|u_0\|_{H^s(\mathbb{R}^3)}+\|v_0-\eta\|_{L^\infty(\mathbb{R}^3)}+\|\nabla v_0\|_{H^s(\mathbb{R}^3)}\le \varepsilon, $$
there is a  solution $(u,p,v)$, unique up to additive functions $\rho(t)$ on the pressure term $p$, for the Ericksen-Leslie system \eqref{EL.sys.} in $\mathbb{R}_+$ with
 $$ u\in X^s,\quad v\in\Theta^s,\quad \nabla p\in L^2\left(\mathbb{R}_+;H^s\left(\mathbb{R}^3;\mathbb{R}^3\right)\right),\quad p(t)\in\widehat{H}^1\left(\mathbb{R}^3\right)\quad \text{for a.e.}\:\:t>0, $$
 where $X^s$ and $\Theta^s$ are defined in \eqref{def.X} and \eqref{def.T} and
    $$ \|u\|_{X^s}+\|\nabla p\|_{L^2(\mathbb{R}_+;H^s(\mathbb{R}^3))}+\|v-\eta\|_{\Theta^s}\le C_s\varepsilon. $$
 Moreover
    $$ \|u(t)\|_{L^\infty(\mathbb{R}^3)}+\|\nabla v(t)\|_{L^\infty(\mathbb{R}^3)}\le C_s t^{-3/4}\quad \text{as}\:\:t\to+\infty. $$
\end{thm}
Concerning the global theorems, we have proved a stability result: if $v_0\in B_\varepsilon(\eta)\subseteq\mathbb{R}^3$ with $\varepsilon\ll1$, then the solution $v$ remains in a ball with a comparable radius.

\vspace{2mm}

 The paper is organized as follows: in the second section we state and prove some linear estimates we will need for the proof of Theorem \ref{t.loc.ex.} and \ref{t.gl-dec.ex.}, which are proven respectively in Section 3 and 4. In the last section we turn back to the Ericksen-Leslie system and we prove Theorems \ref{t.ex.int.} and \ref{t.ex.fr.}.

\section{Linear Estimates}

In this section we consider the linear system
\begin{equation}
    \left\{\begin{array}{ll}
        (\partial_t-\Delta)u+\nabla p=f & \mathbb{R}_+\times \mathbb{R}^N \\
        {\rm div}u=0 & \mathbb{R}_+\times\mathbb{R}^N \\
        (\partial_t-\Delta)d=g & \mathbb{R}_+\times \mathbb{R}^N \\
        u(0)=u_0,\quad d(0)=d_0 & \mathbb{R}^N.
    \end{array}\right.
\end{equation}
If we call $\mathbb{P}\colon L^q(\mathbb{R}^N;\mathbb{R}^N)\to L^q(\mathbb{R}^N;\mathbb{R}^N)$ for $q\in(1,\infty)$ the Leray projection, then we can rewrite the system as follows:
\begin{equation}\label{proj.sys.}
    \left\{\begin{array}{ll}
        (\partial_t-\mathbb{P}\Delta)u=f & \mathbb{R}_+\times \mathbb{R}^N \\
        (\partial_t-\Delta)d=g & \mathbb{R}_+\times \mathbb{R}^N \\
        u(0)=u_0,\quad d(0)=d_0 & \mathbb{R}^N,
    \end{array}\right.
\end{equation}
with $g\in L^q(\mathbb{R}^N)$ and $f\in J_q(\mathbb{R}^N)$. The system \eqref{proj.sys.} is diagonal on the variables $(u,d)$. Precisely, we have a Stokes equation on $u$ and a heat equation on $d$, so we can write explicitly the solution for \eqref{proj.sys.} using Duhamel formula:
\begin{equation}\label{Duh.for.}
    \begin{aligned}
     & u(t)=e^{\mathbb{P}\Delta t}u_0+\int_0^te^{\mathbb{P}\Delta(t-\tau)}f(\tau)d\tau\quad \text{for a.e.}\:\:t>0, \\
     & d(t)=e^{\Delta t}d_0+\int_0^te^{\Delta(t-\tau)}g(\tau)d\tau\quad \text{for a.e.}\:\:t>0,
\end{aligned}
\end{equation}
where $\{e^{\Delta t}\}_t$ and $\{e^{\mathbb{P}\Delta}\}_t$ are respectively the heat and the Stokes semigroups. Our aim is to find some estimates for these functions: it is well-known that
$$ e^{\Delta t}f(x)=\int_{\mathbb{R}^N}\frac{e^{-\frac{|x-y|^2}{4t}}}{(4\pi t)^{N/2}}f(y)dy. $$
\begin{lem}\label{l.est.HS}
Let $q,r\in[1,\infty]$ with $r\le q$, let $k\in\mathbb{N}$, then
$$ \|\nabla^k e^{\Delta t}f\|_{L^q(\mathbb{R}^N)}\le C(N,q,r,k) t^{-\frac{N}{2}\left(\frac{1}{r}-\frac{1}{q}\right)-\frac{k}{2}}\|f\|_{L^r(\mathbb{R}^N)}\quad \forall t>0. $$
Moreover, if $1<r<q\le \infty$ or $1<r\le q<\infty$, then
$$ \|\nabla^k e^{\mathbb{P}\Delta t}\mathbb{P}f\|_{L^q(\mathbb{R}^N)}\le C(N,q,r,k) t^{-\frac{N}{2}\left(\frac{1}{r}-\frac{1}{q}\right)-\frac{k}{2}}\|f\|_{L^r(\mathbb{R}^N)}\quad \forall t>0. $$
\end{lem}
 If we come back to the Duhamel formula \eqref{Duh.for.}, the previous lemma gives us an estimate for the first term. For what concerns the other term, we can prove the following smoothing estimates:
\begin{lem}\label{l.sm.es.}
For any $m,s \geq 0$, $j=0,1$ and any $T \in (0,\infty]$ we have
\begin{equation}\label{sm.es.1}
   \left\| (-\Delta)^{m/2}  e^{\mathbb{P}^j\Delta t} \mathbb{P}^jw_0 \right\|_{L^\infty((0,T);H^s(\mathbb{R}^N))} \leq C_N \|(-\Delta)^{m/2}w_0\|_{H^s(\mathbb{R}^N)}
\end{equation}
\begin{equation}\label{sm.es.2}
    \begin{aligned}
        \left\|(-\Delta)^{m/2} \int_0^t e^{\mathbb{P}^j\Delta(t-\tau)}\mathbb{P}^jh(\tau,\cdot)d\tau \right\|_{L^\infty((0,T);H^s(\mathbb{R}^N))} \leq C_N\|(-\Delta)^{m/2-1/2}h\|_{L^2((0,T);H^s(\mathbb{R}^N))}
    \end{aligned}
\end{equation}
\begin{equation}\label{sm.es.3}
   \left\| (-\Delta)^{m/2+1/2}  e^{\mathbb{P}^j\Delta t} \mathbb{P}^jw_0 \right\|_{L^2((0,T);H^s(\mathbb{R}^N))} \leq C_N \|(-\Delta)^{m/2}w_0\|_{H^s(\mathbb{R}^N)}
\end{equation}
\begin{equation}\label{sm.es.4}
    \begin{aligned}
        \left\|(-\Delta)^{m/2+1/2} \int_0^t e^{\mathbb{P}^j\Delta(t-\tau)}\mathbb{P}^jh(\tau,\cdot)d\tau \right\|_{L^2((0,T);H^s(\mathbb{R}^N))} \leq C_N\|(-\Delta)^{m/2-1/2}h\|_{L^2((0,T);H^s(\mathbb{R}^N))}.
    \end{aligned}
\end{equation}
\end{lem}
\begin{proof}\hfill\\
It is sufficient to focus on the heat case with  $s=m=0$. By density, we can assume some regularity on the data, that is $w_0 \in \mathscr{S}(\mathbb{R}^N)$, $h \in \mathscr{S}(\mathbb{R} \times \mathbb{R}^N)$.
Then the function $v$ defined as
$$ w = e^{\Delta t} w_0 +  \int_0^t e^{\Delta(t-\tau)}h(\tau)d\tau$$
satisfies again
$$
\frac{1}{2}\|w(t)\|_{L^2(\mathbb{R}^N)}^2+\|\nabla w\|_{L^2((0,t);L^2(\mathbb{R}^N))}^2=\frac{1}{2}\|w_0\|_{L^2(\mathbb{R}^N)}^2+\int_0^t\left<h(\tau),w(\tau)\right>_{L^2_x}d\tau. $$
Now we notice that
$$ \int_0^t\left<h(\tau),w(\tau)\right>_{L^2_x}d\tau\le \int_0^t\|(-\Delta)^{-1/2}h(\tau)\|_{L^2(\mathbb{R}^N)}\|(-\Delta)^{1/2}w(\tau)\|_{L^2(\mathbb{R}^N)}d\tau\le $$
$$ \le \|(-\Delta)^{-1/2}h\|_{L^2((0,T);L^2(\mathbb{R}^N))}\|(-\Delta)^{1/2}w\|_{L^2((0,T);L^2(\mathbb{R}^N))}. $$
so we deduce
$$ \|w\|_{L^\infty((0,T);L^2(\mathbb{R}^N))}+\|(-\Delta)^{1/2} w\|_{L^2((0,T);L^2(\mathbb{R}^N))} \lesssim  \|w_0\|_{L^2(\mathbb{R}^N)} + \|(-\Delta)^{-1/2} h\|_{L^2((0,T);L^2(\mathbb{R}^N))}.
$$
In particular, when $h=0$ we get \eqref{sm.es.1} and \eqref{sm.es.3}. Otherwise, when $w_0=0$ we get \eqref{sm.es.2} and \eqref{sm.es.4}.
\end{proof}
Thanks to Lemma \ref{l.est.HS} and \ref{l.sm.es.} we can prove the linear estimate for  \eqref{proj.sys.}. Before, we state the following result from \cite{GN18}.
\begin{thm}\label{t.w.N.}\hfill\\
    Let $q\in(1,\infty)$ and let $f\in L^q(\mathbb{R}^N;\mathbb{R}^N)$, then we can find a unique solution $[\pi]\in\widehat{W}^{1,q}(\mathbb{R}^N)$ for the weak problem
    $$ \Delta\pi={\rm div}f\quad \text{on}\:\:\mathbb{R}^N, $$
    moreover
    $$ \|\nabla \pi\|_{L^q(\mathbb{R}^N)}\le C(q,N)\|f\|_{L^q(\mathbb{R}^N)}. $$
\end{thm}
By the definition of $\widehat{W}^{1,q}(\mathbb{R}^N)$ we get then the choice of $\pi$ in the previous theorem is unique up to a constant $c\in\mathbb{R}$.
\begin{cor}\label{c.lin.ex.}
Let $N\ge 2$, $s\ge 0$,
$$ u_0\in J_2\left(\mathbb{R}^N\right)\cap H^s\left(\mathbb{R}^N;\mathbb{R}^N\right), \quad d_0\colon\mathbb{R}^N\to\mathbb{R},\quad\nabla d_0\in H^{s}\left(\mathbb{R}^N;\mathbb{R}^N\right),  $$
let $T\in(0,\infty]$ and
$$ F\in L^2\left((0,T);H^s\left(\mathbb{R}^N;\mathbb{R}^{N^2}\right)\right), \quad g\in L^2\left((0,T);H^{s}\left(\mathbb{R}^N\right)\right), $$
then there is a solution $(u,p,d)$, unique up to additive functions $\rho(t)$ on the pressure term $p$, for the linear system
$$
\left\{\begin{array}{ll}
    (\partial_t-\Delta)u+\nabla p={\rm Div}F & (0,T)\times\mathbb{R}^N \\
    {\rm div}u=0 & (0,T)\times\mathbb{R}^N \\
    (\partial_t-\Delta)d=g & (0,T)\times \mathbb{R}^N \\
    u(0)=u_0,\quad d(0)=d_0 & \mathbb{R}^N,
\end{array}\right. $$
with
$$ u\in X^s_T,\quad \nabla d\in X^s_T,\quad \nabla p\in L^2\left((0,T);H^s\left(\mathbb{R}^N\right)\right),\quad p(t)\in\widehat{H}^1\left(\mathbb{R}^N\right)\quad \text{for a.e.}\:\:t\in(0,T) $$
where $X^s_T$ and $\Theta^s_T$ are defined in \eqref{def.X} and \eqref{def.T} and
$$  \|u\|_{X^s_T}+ \|\nabla p\|_{L^2((0,T);H^s(\mathbb{R}^N))}+ \| \nabla d\|_{X^s_T}\le $$
$$ \le C_N\left[\|u_0\|_{H^s(\mathbb{R}^N)}+\|\nabla d_0\|_{H^{s}(\mathbb{R}^N)} + \|F\|_{L^2((0,T);H^s(\mathbb{R}^N))} + \|g\|_{L^2((0,T);H^s(\mathbb{R}^N))}\right]. $$
\end{cor}
\begin{proof}\hfill\\
From the system \eqref{proj.sys.} we get the Duhamel formula satisfied by the solution $u$ and $d$:
$$ u=e^{\mathbb{P}\Delta t}u_0+\int_0^te^{\mathbb{P}\Delta (t-\tau)}\mathbb{P}{\rm Div}F(\tau)d\tau $$
$$ d=e^{\Delta t}d_0+\int_0^t e^{\Delta(t-\tau)}g(\tau)d\tau. $$
Let us start with the estimate for $u$: from Lemma \ref{l.sm.es.} with $m=0$ we have that
$$ \|u\|_{L^\infty((0,T);H^s(\mathbb{R}^N))}+\|\nabla u\|_{L^2((0,T);H^s(\mathbb{R}^N))}\lesssim \|u_0\|_{H^s(\mathbb{R}^N)}+\|(-\Delta)^{-1/2}{\rm Div}F\|_{L^2((0,T);H^s(\mathbb{R}^N))}. $$
So
$$ \|u\|_{X^s_T}\lesssim \|u_0\|_{H^s(\mathbb{R}^N)}+\|F\|_{L^2((0,T);H^s(\mathbb{R}^N))}. $$
Moreover, from \eqref{sm.es.1} and \eqref{sm.es.3} with $m=1$ we have that
$$ \|\nabla e^{\Delta t}d_0\|_{L^\infty((0,T);H^s(\mathbb{R}^N))}+\|\nabla^2e^{\Delta t}d_0\|_{L^2((0,T);H^s(\mathbb{R}^N))}\lesssim \|\nabla d_0\|_{H^s(\mathbb{R}^N)}. $$
On the other hand, thanks to \eqref{sm.es.2} and \eqref{sm.es.4} with $m=1$, we have that
$$ \left\|\nabla \int_0^te^{\Delta(t-\tau)}g(\tau)d\tau\right\|_{L^\infty((0,T);H^s(\mathbb{R}^N))}+\left\|\nabla^2\int_0^te^{\Delta(t-\tau)}g(\tau)d\tau\right\|_{L^2((0,T);H^s(\mathbb{R}^N))}\lesssim $$
$$ \lesssim \|g\|_{L^2((0,T);H^s(\mathbb{R}^N))}. $$
In conclusion
$$
    \|\nabla d\|_{X^s_T}\lesssim \|\nabla d_0\|_{H^{s}(\mathbb{R}^N)}+\|g\|_{L^2((0,T);H^s(\mathbb{R}^N))}. $$
On the other hand ${\rm div}u(t)=0$ for a.e. $t\in(0,T)$, because $w\coloneqq {\rm div}u$ solves the Cauchy problem
$$ \left\{\begin{array}{ll}
\partial_tw=0 & (0,T)\times \mathbb{R}^N \\
w(0)=0 & \mathbb{R}^N,
\end{array}\right.$$
so $w={\rm div}u=0$. Finally, if apply $(1-\Delta)^{s/2}$ in the first equation of \eqref{EL.red.sys.} and then we take the divergence, we get that $p$ solves weakly for a.e. $t\in(0,T)$ the problem
$$ \Delta \left[(1-\Delta)^{s/2}p(t)\right] ={\rm div}\left[(1-\Delta)^{s/2}{\rm Div}F(t)\right] \quad \text{on}\:\:\mathbb{R}^N\quad \text{for a.e.}\:\:t\in(0,T). $$
Since ${\rm Div}F(t)\in H^s(\mathbb{R}^N;\mathbb{R}^N)$ for a.e. $t\in (0,T)$, by Theorem \ref{t.w.N.} we get that $p(t)\in \widehat{H}^1(\mathbb{R}^N)$ for a.e. $t\in(0,T)$ and
$$ \|\nabla p(t)\|_{H^s(\mathbb{R}^N)}\lesssim \|{\rm Div}F\|_{H^s(\mathbb{R}^N)}. $$
Taking the $L^2((0,T))$-norm of the previous estimate we conclude.
\end{proof}
 We conclude this paragraph with some well-known decay estimates for the Stokes and the heat semigroups that will be helpful in the next sections:
\begin{lem}\label{l.decay}
Let $N\ge 2$ and $s>0$ then
    $$ \begin{aligned}
    & \|e^{\Delta t}f\|_{L^\infty(\mathbb{R}^N)}\le C(N,s)\left(t^{\alpha}+t^{N/4}\right)^{-1}\|f\|_{H^s(\mathbb{R}^N)} \\
    & \|e^{\mathbb{P}\Delta t}\mathbb{P}f\|_{L^\infty(\mathbb{R}^N)}\le C(N,s) \left(t^{\alpha}+t^{N/4}\right)^{-1}\|f\|_{H^s(\mathbb{R}^N)},
\end{aligned} $$
with
\begin{equation}\label{a.value}
    \alpha= \left\{\begin{array}{ll}
    \frac{N}{4}-\frac{s}{2} & s<\frac{N}{2} \\
    \delta & s=\frac{N}{2} \\
    0 & s>\frac{N}{2}
\end{array}\right.
\end{equation}
with $\delta\in \left(0,\frac{N}{4}\right)$ arbitrary small.
\end{lem}
\begin{rem}\label{rem.alpha}
   It is useful to notice for the proceeding of the paper that $\alpha<\frac{1}{2}$ when $s>\frac{N}{2}-1$ for any choice of $\alpha$ from \eqref{a.value}.
\end{rem}
The proof of Lemma \ref{l.decay} follows from Lemma \ref{l.est.HS} and Sobolev embeddings.

\section{Local Existence}

Let us consider now the system
\begin{equation}\label{s.loc.nl.}
    \left\{\begin{array}{ll}
        (\partial_t-\Delta)u+\nabla p+u\cdot \nabla u=-{\rm Div}(\nabla d\otimes\nabla d) & (0,T)\times \mathbb{R}^N \\
        {\rm div}u=0 & (0,T)\times\mathbb{R}^N\\
        (\partial_t-\Delta)d+u\cdot \nabla d=0 & (0,T)\times \mathbb{R}^N \\
        u(0)=u_0,\quad d(0)=d_0 & \mathbb{R}^N.
    \end{array}\right.
\end{equation}
The aim of this section is to prove the existence of a solution for the system \eqref{s.loc.nl.}. We recall the statement:
\begin{thm}\label{t.loc.ex.+}\hfill\\
Let $N\ge 3$, $s>\frac{N}{2}-1$. For any $R>0$ there exists $T=T(R)>0$ so that for any initial data
$$ u_0\in J_2\left(\mathbb{R}^N\right)\cap H^s\left(\mathbb{R}^N;\mathbb{R}^N\right), \quad d_0\in L^\infty\left(\mathbb{R}^N\right),\quad \nabla d_0\in H^{s}\left(\mathbb{R}^N;\mathbb{R}^N\right) $$
and
$$ \|u_0\|_{H^s(\mathbb{R}^N)}+\|d_0\|_{L^\infty(\mathbb{R}^N)}+\|\nabla d_0\|_{H^{s}(\mathbb{R}^N)} \leq R$$
there exists a solution $(u,p,d)$, unique up to additive functions $\rho(t)$ on the pressure term $p$, for the reduced system \eqref{s.loc.nl.} with
$$ u\in X^s_T,\quad d\in\Theta^s_T,\quad \nabla p\in L^2\left((0,T);H^s\left(\mathbb{R}^N;\mathbb{R}^N\right)\right),\quad p(t)\in\widehat{H}^1\left(\mathbb{R}^N\right)\quad \text{for a.e.}\:\:t\in(0,T),$$
where $X^{s}_T$ and $\Theta^s_T$ are defined in \eqref{def.X} and \eqref{def.T}, and
$$ \|u\|_{X^s_T}+\|\nabla p\|_{L^2((0,T);H^s(\mathbb{R}^N))}+\|d\|_{\Theta^s_T}\le C(N,s)\left[\|u_0\|_{H^s(\mathbb{R}^N)}+\|d_0\|_{L^\infty(\mathbb{R}^N)}+\|\nabla d_0\|_{H^{s}(\mathbb{R}^N)}\right]. $$
\end{thm}
The strategy is to use a contraction argument for the map $\phi\colon X^s_T\times\Theta^s_T\to X^s_T\times\Theta^s_T$ such that $\phi(w,\theta)=(u,d)$ solves
\begin{equation}\label{s.aux.loc.}
    \left\{\begin{array}{ll}
    (\partial_t-\mathbb{P}\Delta)u=-\mathbb{P}(w\cdot \nabla w -{\rm Div}(\nabla \theta\odot\nabla \theta)) & (0,T)\times\mathbb{R}^N \\
    (\partial_t-\Delta)d=-w\cdot \nabla \theta & (0,T)\times\mathbb{R}^N \\
    u(0)=u_0,\quad d(0)=d_0 & \mathbb{R}^N.
    \end{array}\right.
\end{equation}
The existence of such a function $\phi$ can be proven by the linear estimate we gained in Corollary \ref{c.lin.ex.}. On the other hand, in order to apply Corollary \ref{c.lin.ex.} we need to prove that the right hand side of the system \eqref{s.aux.loc.} satisfies the condition we gave in the statement. Firstly, we notice that, if we ask ${\rm div}w=0$, we have that
$$ w\cdot \nabla w +{\rm Div}(\nabla \theta\odot\nabla \theta)={\rm Div}\left(w\odot w+\nabla \theta\odot \nabla \theta\right), $$
where we recall that
$$ Div A=\sum_j\partial_j A^j\quad \forall A\colon\mathbb{R}^N\to\mathbb{R}^{N^2}. $$
Therefore, what it remains to do in order to apply Corollary \ref{c.lin.ex.} is to prove some estimates for the nonlinear terms of \eqref{s.aux.loc.}.

We will need the following fractional Leibniz rule in $H^s(\mathbb{R}^N)$:
\begin{thm}[Theorem 1.4 \cite{GK96}]\hfill\\
Let $s>0$, $r\in(1,\infty)$, let $1<p_1,q_2<\infty$ and $1<p_2,q_1\le  \infty$ such that
$$ \frac{1}{r}=\frac{1}{p_1}+\frac{1}{p_2}=\frac{1}{q_1}+\frac{1}{q_2}, $$
then
\begin{equation}\label{fr.L.}
    \|fg\|_{H^s_r(\mathbb{R}^N)}\le \|f\|_{H^s_{p_1}(\mathbb{R}^N)}\|g\|_{L^{p_2}(\mathbb{R}^N)}+\|f\|_{L^{q_1}(\mathbb{R}^N)}\|g\|_{H^s_{q_2}(\mathbb{R}^N)}.
\end{equation}
\end{thm}
 We are now ready to prove some bilinear estimates that we will use to bound the right hand side of \eqref{s.aux.loc.}.
\begin{lem}
Let $T>0$, $s>\frac{N}{2}-1$ and $z,w\in X^{s}_{T}$, then we can find $\gamma=\gamma(N,s)\in(0,1)$ such that
\begin{equation}\label{bil.e.1}
    \|zw\|_{L^2((0,T);H^s(\mathbb{R}^N))} \le C(N,s)  T^\gamma \|z\|_{X^{s}_{T}}\|w\|_{X^{s}_{T}}.
    \end{equation}
\end{lem}
\begin{proof} \hfill\\
By the fractional Leibniz rule \eqref{fr.L.}:
$$ \|z(\tau)w(\tau)\|_{H^s(\mathbb{R}^N)}\lesssim \|z(\tau)\|_{H^s_{p_1}(\mathbb{R}^N)}\|w(\tau)\|_{L^{q_1}(\mathbb{R}^N)} + \|z(\tau)\|_{L^{p_2}(\mathbb{R}^N)}\|w(\tau)\|_{H^s_{q_2}(\mathbb{R}^N)}\quad \forall \tau\in(0,T) $$
for
$$ \frac{1}{2}=\frac{1}{p_1}+\frac{1}{q_1}=\frac{1}{p_2}+\frac{1}{q_2}, \quad p_1,q_2<\infty. $$
We focus just on the first term, the other is analogous. We divide the proof in two cases: if $s<\frac{N}{2}$, the we can take $q_1=\frac{2N}{N-2s}$, so by Sobolev embedding we have that
$$ \|w(\tau)\|_{L^{q_1}(\mathbb{R}^N)}\lesssim \|w(\tau)\|_{H^s(\mathbb{R}^N)}\le \|w\|_{L^\infty((0,T);H^s(\mathbb{R}^N))}\le \|w\|_{X^s_T}. $$
On the other hand $p_1=\frac{N}{s}<\frac{2N}{N-2}$ as $s>\frac{N}{2}-1$. Therefore, we can find $\theta\in(0,1)$ such that
$$ \|z(\tau)\|_{L^{p_1}(\mathbb{R}^N)}\le \|z(\tau)\|_{L^2(\mathbb{R}^N)}^\theta\|z(\tau)\|_{L^\frac{2N}{N-2}(\mathbb{R}^N)}^{1-\theta} \lesssim $$
$$ \lesssim \|z\|_{L^\infty((0,T);L^2(\mathbb{R}^N))}^\theta\|\nabla z(\tau)\|_{L^2(\mathbb{R}^N)}^{1-\theta}\le \|z\|^\theta_{X^s_T}\|\nabla z(\tau)\|_{L^2(\mathbb{R}^N)}^{1-\theta}, $$
where in the last inequality we have used Sobolev embedding. Finally
$$ \|zw\|_{L^2((0,T);H^s(\mathbb{R}^N))} \lesssim T^\theta\|w\|_{X^s_T}\|z\|^\theta_{X^s_T}\|\nabla z\|_{L^2((0,T);L^2(\mathbb{R}^N))}^{1-\theta}\le T^\theta\|z\|_{X^s_T}\|w\|_{X^s_T}. $$
The case $s\ge \frac{N}{2}$ is easier: by Sobolev embedding we have that for any $q_1\in[2,\infty)$
$$ \|w(\tau)\|_{L^{q_1}(\mathbb{R}^N)}\lesssim \|w(\tau)\|_{H^s(\mathbb{R}^N)}\le \|w\|_{L^\infty((0,T);H^s(\mathbb{R}^N))}\le \|w\|_{X^s_T}. $$
Therefore, it is sufficient to take $p_1\in\left(2,\frac{2N}{N-2}\right)$ such that there is $\theta\in(0,1)$ as before which satisfies the inequality
$$ \|z(\tau)\|_{L^{p_1}(\mathbb{R}^N)}\le \|z(\tau)\|_{L^2(\mathbb{R}^N)}^\theta\|z(\tau)\|_{L^{\frac{2N}{N-2}}(\mathbb{R}^N)}^{1-\theta}. $$
Then we conclude as before.
\end{proof}
Now we want to prove the $L^\infty$ estimate for the $d$-term. In order to do so, we will use the Duhamel formula:
$$ d=e^{\Delta t}d_0+\int_0^te^{\Delta(t-\tau)}(u\cdot\nabla d)(\tau)d\tau. $$
Since $d_0\in L^\infty(\mathbb{R}^N)$, the estimate for the first term comes from Lemma \ref{l.est.HS}. Therefore, we can focus on the second term:
\begin{lem}
 Let $N\ge 3$, $s>\frac{N}{2}-1$, $T>0$, $z,w\in X^s_T$ then there is $\gamma=\gamma(N,s)>0$ such that
\begin{equation}\label{bil.e.2}
    \left\|\int_0^t e^{\Delta(t-\tau)} z(\tau)w(\tau)d\tau\right\|_{L^\infty((0,T);L^\infty(\mathbb{R}^N))}\le C(N,s)T^\gamma \|z\|_{X^s_T}\|w\|_{X^s_T};
\end{equation}
\end{lem}
\begin{proof}\hfill\\
Let $r>\frac{N}{2}$, then from Lemma \ref{l.est.HS}
$$ \left|\int_0^t e^{\Delta(t-\tau)}z(\tau)w(\tau)d\tau\right|\lesssim \int_0^t(t-\tau)^{-\frac{N}{2r}}\|z(\tau)w(\tau)\|_{L^r(\mathbb{R}^N)}d\tau. $$
In order to conclude, it is then sufficient to prove that
$$ \|zw\|_{L^\infty((0,T);L^r(\mathbb{R}^N))}\lesssim \|z\|_{X^s_T}\|w\|_{X^s_T}, $$
for some $r>\frac{N}{2}$.
$$ \|z(\tau)w(\tau)\|_{L^r(\mathbb{R}^N)}\le\|z(\tau)\|_{L^{r_1}(\mathbb{R}^N)}\|w(\tau)\|_{L^{r_2}(\mathbb{R}^N)}, $$
with
$$ \frac{1}{r_1}+\frac{1}{r_2}=\frac{1}{r}. $$
If $s\ge \frac{N}{2}$ by Sobolev embedding any choice of $r_1,r_2\in\left(\frac{N}{2},\infty\right)$ works. Let us suppose then $s<\frac{N}{2}$. If we take $r_1=r_2=\frac{2N}{N-2s}$, then by Sobolev embedding
$$ \|z(\tau)\|_{L^{r_1}(\mathbb{R}^N)}\|w(\tau)\|_{L^{r_2}(\mathbb{R}^N)}\lesssim \|z(\tau)\|_{H^s(\mathbb{R}^N)}\|w(\tau)\|_{H^s(\mathbb{R}^N)}\le \|z\|_{X^s}\|w\|_{X^s}. $$
On the other hand, if $r_1=r_2=\frac{2N}{N-2s}$, then $r=\frac{N}{N-2s}>\frac{N}{2}$ since $s>\frac{N}{2}-1$.
\end{proof}

Thanks to this bilinear estimate, we can now prove the local existence theorem:
\begin{proof}[Proof of Theorem \ref{t.loc.ex.+}]\hfill\\
Let us define the space
$$ Y_T\coloneqq \{(w,\theta)\in X^s_T\times\Theta^s_T\mid {\rm div}(w(t))=0\:\:\text{for a.e.}\:\:t\in(0,T), \:\:\|(w,\theta)\|_{Y_T}<+\infty\}  $$
where
$$ \|(w,\theta)\|_{Y_T}\coloneqq\|w\|_{X^s_T}+\|\theta\|_{\Theta^s_T}. $$
Now we define the map $\phi\colon Y_T\to Y_T$ such that $\phi(w,\theta)=(u,d)$ solves
\begin{equation}
    \left\{\begin{array}{ll}
    (\partial_t-\mathbb{P}\Delta)u=-\mathbb{P}(w\cdot \nabla w +{\rm Div}(\nabla \theta\odot\nabla \theta)) & (0,T)\times\mathbb{R}^N \\
    (\partial_t-\Delta)d=-w\cdot \nabla \theta & (0,T)\times\mathbb{R}^N \\
    u(0)=u_0,\quad d(0)=d_0 & \mathbb{R}^N.
    \end{array}\right.
\end{equation}
Namely,
\begin{equation}\label{Duh.f.loc.}
    \begin{aligned}
        & u=e^{\mathbb{P}\Delta t}u_0+\int_0^te^{\mathbb{P}\Delta(t-\tau)}\mathbb{P}f(w,\theta)(\tau)d\tau \\
        & d=e^{\Delta t}d_0+\int_0^t e^{\Delta(t-\tau)}g(w,\theta)(\tau)d\tau,
    \end{aligned}
\end{equation}
where
$$ f(w,\theta)=-w\cdot \nabla w-{\rm Div}(\nabla \theta\odot\nabla\theta), $$
$$ g(w,\theta)=w\cdot \nabla \theta. $$
The plan of the proof is to prove that $\phi$ is a contraction in a suitable subspace of $Y_T$. Firstly, we will prove the following estimates:
\begin{equation}
    \|\phi(w,\theta)\|_{Y_T} \leq C \left( \|u_0\|_{H^s(\mathbb{R}^N)} + \|d_0\|_{L^\infty(\mathbb{R}^N)}+\|\nabla d_0\|_{H^{s}(\mathbb{R}^N)} \right) + C T^\gamma \|(w,\theta)\|_{Y_T}^2
\end{equation}
\begin{equation}
    \|\phi(w_1,\theta_1)- \phi(w_2,\theta_2)\|_{Y_T} \leq  MT^\gamma  \left[\|(w_1,\theta_1)\|_{Y_T} + \|(w_2,\theta_2)\|_{Y_T}\right]\|(w_1,\theta_1)-(w_2,\theta_2)\|_{Y_T}
\end{equation}
for some $\gamma,C,M>0$.

\textbf{Step 1}: Thanks to the divergence free condition for $w$, we have that
$$ w\cdot \nabla w={\rm Div}(w\odot w). $$
So by Corollary \ref{c.lin.ex.} we have
$$ \|u\|_{X^s_T}+\|\nabla d\|_{X^{s}_T}\lesssim $$
$$ \lesssim \|w\odot w\|_{L^2((0,T);H^s(\mathbb{R}^N))}+\|\nabla \theta\odot\nabla \theta\|_{L^2((0,T);H^s(\mathbb{R}^N))} +\|w\cdot \nabla \theta\|_{L^2((0,T);H^s(\mathbb{R}^N))}. $$
From \eqref{bil.e.1} we know that exists $\gamma>0$ such that
\begin{equation}\label{ww.loc.}
    \|w\odot w\|_{L^2((0,T);H^s(\mathbb{R}^N))}\lesssim T^\gamma \|w\|_{X^s_T}^2.
\end{equation}
\begin{equation}\label{DtDt.loc.}
    \|\nabla \theta\odot \nabla \theta\|_{L^2((0,T);H^s(\mathbb{R}^N))}\lesssim T^\gamma\|\nabla \theta\|_{X^s_T}^2,
\end{equation}
\begin{equation}\label{wDt.loc.}
    \|w\cdot \nabla \theta\|_{L^2((0,T);H^s(\mathbb{R}^N))}\lesssim T^\gamma\|w\|_{X^s_T}\|\nabla \theta\|_{X^s_T},
\end{equation}
On the other hand, from Lemma \ref{l.est.HS} and the estimate \eqref{bil.e.2} applied at the Duhamel formula \eqref{Duh.f.loc.} we also have that
\begin{equation}\label{Linf.loc.}
    \|d\|_{L^\infty((0,T);L^\infty(\mathbb{R}^N)}\lesssim \|d_0\|_{L^\infty(\mathbb{R}^N)}+T^\gamma \|w\|_{X^s_T}\|\nabla \theta\|_{X^s_T}.
\end{equation}

Thanks to \eqref{ww.loc.}, \eqref{DtDt.loc.},  \eqref{wDt.loc.} with Corollary \ref{c.lin.ex.} and from \eqref{Linf.loc.} we get that
\begin{equation}\label{nl.loc.1}
    \|\phi(w,\theta)\|_{Y_T}=\|u\|_{X^s_T}+\|d\|_{\Theta^{s}_T}\lesssim  \|u_0\|_{H^s(\mathbb{R}^N)}+\|d_0\|_{L^\infty(\mathbb{R}^N)}+\|\nabla d_0\|_{H^{s}(\mathbb{R}^N)} + T^\gamma\|(w,\theta)\|_{Y_T}^2
\end{equation}
for $\gamma>0$.
In a similar way it can be proven that
\begin{equation}\label{nl.loc.2}
    \|\phi(w_1,\theta_1)-\phi(w_2,\theta_2)\|_{Y_T}\lesssim   T^\gamma\left(\|(w_1,\theta_1)\|_{Y_T}+\|(w_2,\theta_2)\|_{Y_T}\right)\|(w_1,\theta_1)-(w_2,\theta_2)\|_{Y_T}.
\end{equation}

\textbf{Step 2}: Let us define now the space
$$ Z_\omega\coloneqq \{(w,\theta)\in Y_T\mid w(0)=u_0,\quad \theta(0)=d_0,\quad \|(w,\theta)\|_{Y_T}\le \omega\} $$
for $\omega>0$ to be defined. Let us see that $\phi\colon Z_\omega\to Z_\omega$ is a contraction for a suitable choice of $\omega$: thanks to \eqref{nl.loc.1} we know that exists $C>0$ such that
$$ \|\phi(w,\theta)\|_{Y_T}\le C\left[\|u_0\|_{H^s(\mathbb{R}^N)}+\|d_0\|_{L^\infty(\mathbb{R}^N)}+\|\nabla d_0\|_{H^{s}(\mathbb{R}^N)}+T^\gamma\omega^2\right]\le C[R+T^\gamma\omega^2], $$
so, if we choose $\omega$ such that
$$ CR\le \frac{\omega}{2} $$
and $T>0$ sufficiently small, then we have $\phi(w,\theta)\in Z_\omega$. On the other hand, thanks to \eqref{nl.loc.2} there is $M>0$ such that, for any $(w_1,\theta_1),(w_2,\theta_2)\in Z_\omega$, it holds
$$ \|\phi(w_1,\theta_1)-\phi(w_2,\theta_2)\|_{Y_T}\le 2M\omega T^\gamma \|(w_1,\theta_1)-(w_2,\theta_2)\|_{Y_T}. $$
Therefore, choosing $T>0$ sufficiently small, we get that $\phi\colon Z_\omega\to Z_\omega$ is a contraction and we get the thesis.
\end{proof}

\section{Global Existence and Decay in Time}
\medskip

We are now interested on the system on all $\mathbb{R}_+$:
\begin{equation}\label{s.gl.nl.}
    \left\{\begin{array}{ll}
        (\partial_t-\Delta)u+\nabla p+u\cdot \nabla u=-{\rm Div}(\nabla d\otimes\nabla d) & \mathbb{R}_+\times \mathbb{R}^N \\
        {\rm div}u=0 & \mathbb{R}_+\times\mathbb{R}^N \\
        (\partial_t-\Delta)d+u\cdot \nabla d=0 & \mathbb{R}_+\times \mathbb{R}^N \\
        u(0)=u_0,\quad d(0)=d_0 & \mathbb{R}^N.
    \end{array}\right.
\end{equation}
In the previous section we proved the existence of a solution on the interval $(0,T)$ for $T>0$ sufficiently small. Now we want to prove that, for small initial data, the solution can be found for every $t\in\mathbb{R}_+$ and, moreover, it decays in time. More precisely, we will prove the following result:
\begin{thm}\label{t.gl-dec.ex.+}
  Let $N\ge 3$ and $s>\frac{N}{2}-1$. Then there exists $\varepsilon_0>0$ so that for any $\varepsilon \in (0,\varepsilon_0]$ and any initial data
$$ u_0\in J_2\left(\mathbb{R}^N\right)\cap  H^s\left(\mathbb{R}^N;\mathbb{R}^N\right), \quad  d_0\in L^\infty\left(\mathbb{R}^N\right), \quad \nabla d_0\in H^s\left(\mathbb{R}^N;\mathbb{R}^N\right), $$
such that
$$ \|u_0\|_{H^s(\mathbb{R}^N)}+\|d_0\|_{L^\infty(\mathbb{R}^N)}+\|\nabla d_0\|_{H^{s}(\mathbb{R}^N)}\le \varepsilon, $$
 there is a  solution $(u,p,d)$, unique up to additive functions $\rho(t)$ on the pressure term $p$, for the reduced system \eqref{s.gl.nl.} with
 $$ u\in X^s,\quad d\in\Theta^s,\quad \nabla p\in L^2\left(\mathbb{R}_+;H^s\left(\mathbb{R}^N;\mathbb{R}^N\right)\right),\quad p(t)\in\widehat{H}^1\left(\mathbb{R}^N\right)\quad \text{for a.e.}\:\:t>0, $$
where $X^{s}$ and $\Theta^s$ are defined in \eqref{def.X} and \eqref{def.T}, and
$$ \|u\|_{X^s}+\|\nabla p\|_{L^2(\mathbb{R}_+;H^s(\mathbb{R}^N))}+\|d\|_{\Theta^s}\le C(N,s)\varepsilon. $$
Moreover, for any $k\in\mathbb{N}$ such that $s-k>\frac{N}{2}-1$, it holds
    $$ \|u(t)\|_{W^{k,\infty}(\mathbb{R}^N)}+\|\nabla d(t)\|_{W^{k,\infty}(\mathbb{R}^N)}\le C(N,s) t^{-N/4}\quad \text{as}\:\:t\to+\infty. $$
\end{thm}
As before, the strategy is to use a contraction argument. This time, we will consider the following Banach spaces:
\begin{defn}\hfill\\
Let $N\ge 3$, $s,\alpha\ge 0$ the we defined
$$ X^{s,\alpha}_k\coloneqq \{w\in X^s\mid \|w\|_{X^{s,\alpha}_k}<+\infty\} $$
$$ \Theta^{s,\alpha}_k\coloneqq \left\{\theta\in L^\infty\left(\mathbb{R}_+;L^\infty\left(\mathbb{R}^N\right)\cap\widehat{H}^1\left(\mathbb{R}^N\right)\right)\:\Big|\: \|\nabla \theta\|_{X^{s,\alpha}_k}<+\infty\right\} $$
with
$$ \|w\|_{X^{s,\alpha}_k}\coloneqq \|w\|_{X^s}+\sup_{t>0}(t^\alpha+t^{N/4})\|w(t)\|_{W^{k,\infty}(\mathbb{R}^N)},
$$
$$ \|\theta\|_{\Theta^{s,\alpha}_k}\coloneqq \|\theta\|_{L^\infty(\mathbb{R}_+;L^\infty(\mathbb{R}^N))}+\|\nabla \theta\|_{X^{s,\alpha}_k}. $$
\end{defn}
The strategy is to prove that, for some choice of $\alpha,s,k$, the map $\phi\colon X^{s,\alpha}_k\times \Theta^{s,\alpha}_k\to X^{s,\alpha}_k\times \Theta^{s,\alpha}_k$ is a contraction on a suitable subset of $X^{s,\alpha}_k\times \Theta^{s,\alpha}_k$, where $\phi(w,\theta)=(u,d)$ solves
\begin{equation}\label{s.aux.gl.}
    \left\{\begin{array}{ll}
    (\partial_t-\mathbb{P}\Delta)u=-\mathbb{P}(w\cdot \nabla w +{\rm Div}(\nabla \theta\odot\nabla \theta)) & \mathbb{R}_+\times\mathbb{R}^N \\
    (\partial_t-\Delta)d=-w\cdot \nabla \theta & \mathbb{R}_+\times\mathbb{R}^N \\
    u(0)=u_0,\quad d(0)=d_0 & \mathbb{R}^N.
    \end{array}\right.
\end{equation}
Again, we want to apply Corollary \ref{c.lin.ex.} with $T=\infty$, so we need to prove a bilinear estimate for the right hand side of \eqref{s.aux.gl.}:
\begin{lem}
Let $s\ge \frac{N}{2}-1$ and $z,w\in X^s$, then
\begin{equation}\label{bil.e.3}
    \|zw\|_{L^2(\mathbb{R}_+;H^s(\mathbb{R}^N))}\le C(N,s)\|z\|_{X^s}\|w\|_{X^s}.
\end{equation}
\end{lem}
\begin{proof}\hfill\\
Thanks to the fractional Leibniz rule \eqref{fr.L.} with $q_2=p_1=\frac{2N}{N-2}$ and $q_1=p_2=N$ we get
$$ \|(z w)(\tau)\|_{H^s(\mathbb{R}^N)}\lesssim \|z(\tau)\|_{H^s_\frac{2N}{N-2}(\mathbb{R}^N)}\|w(\tau)\|_{L^N(\mathbb{R}^N)}+\|w(\tau)\|_{H^s_\frac{2N}{N-2}(\mathbb{R}^N)}\|z(\tau)\|_{L^N(\mathbb{R}^N)}. $$
We see just the first term, the other is analogous. By Sobolev embedding we have that
$$ \|z(\tau)\|_{H^s_\frac{2N}{N-2}(\mathbb{R}^N)}\lesssim \|\nabla z(\tau)\|_{H^s(\mathbb{R}^N)}. $$
On the other hand, for any $s\ge \frac{N}{2}-1$, we have by Sobolev embedding that
$$ \|w(\tau)\|_{L^N(\mathbb{R}^N)}\lesssim \|w(\tau)\|_{H^s(\mathbb{R}^N)}\le \|w\|_{L^\infty(\mathbb{R}_+;H^s(\mathbb{R}^N))}\le \|w\|_{X^s}. $$
Finally, we have that
$$ \|zw\|_{L^2(\mathbb{R}_+;H^s(\mathbb{R}^N))}\lesssim \|w\|_{X^s}\|\nabla z\|_{L^2(\mathbb{R}_+;H^s(\mathbb{R}^N))}\le \|w\|_{X^s}\|z\|_{X^s}. $$
\end{proof}
\begin{rem}
It can be seen that the bilinear estimate works also for $s=\frac{N}{2}-1$. In fact the global existence for small initial data can be proven even in this case (but not the decay in time).
\end{rem}
As for the local existence, we need a bound on the $L^\infty(\mathbb{R}_+;L^\infty(\mathbb{R}^N))$-norm. Again, we start from the Duhamel formula for the \eqref{s.aux.gl.}:
\begin{equation}
    \begin{aligned}
        & u=e^{\mathbb{P}\Delta t}u_0+\int_0^te^{\mathbb{P}\Delta(t-\tau)}\mathbb{P}f(w,\theta)(\tau)d\tau \\
        & d=e^{\Delta t}d_0+\int_0^t e^{\Delta(t-\tau)}g(w,\theta)(\tau)d\tau,
    \end{aligned}
\end{equation}
where
$$ f(w,\theta)=-w\cdot \nabla w-{\rm Div}(\nabla \theta\odot\nabla\theta), $$
$$ g(w,\theta)=w\cdot \nabla \theta. $$
We already know the estimate for the firsts terms from Lemma \ref{l.decay}, so it is sufficient to study the integral term. Firstly, we will need some technical lemmas, whose proofs can be found in the Appendix:
\begin{lem}\label{l.tec.int.1}
Let $0\le\alpha_1,\alpha_2<1$ and $T>0$, then it holds
$$ \int_0^t (t-\tau)^{-\alpha_1}\tau^{-\alpha_2}d\tau\le C(\alpha_1,\alpha_2) t^{-\max\{\alpha_1,\alpha_2\}}T^{1-\min\{\alpha_1,\alpha_2\}} \quad \forall t\in(0,T). $$
\end{lem}
\begin{lem}\label{l.tec.int.2}
Let $0\le\alpha_1,\alpha_2<1$ and $\beta_1,\beta_2> 1$, then it holds
$$ \int_0^t \frac{1}{((t-\tau)^{\alpha_1}+(t-\tau)^{\beta_1})}\frac{1}{(\tau^{\alpha_2}+\tau^{\beta_2})}d\tau\le C(\alpha_1,\alpha_2,\beta_1,\beta_2) \left(t^{\max\{\alpha_1,\alpha_2\}}+t^{\min\{\beta_1,\beta_2\}}\right)^{-1}\quad \forall t>0. $$
\end{lem}
 We are now ready to state the bilinear estimates we are going to use in the proof of Theorem \ref{t.gl-dec.ex.+}: this lemma and the next one are devoted to study the $W^{k,\infty}(\mathbb{R}^N;\mathbb{R}^N)$-norm of $u$.
\begin{lem}\label{l.bil.dec.1}
Let $N\ge 3$, $k\in\mathbb{N}$, $s>\frac{N}{2}-1$ and $z,w\in X^{s,\alpha}_k$ with ${\rm div}z(t)=0$ for a.e. $t>0$, then we can find $\alpha\in\left(0,\frac{1}{2}\right)$ such that
 \begin{equation}
     \left\|\int_0^t e^{\mathbb{P}\Delta(t-\tau)}\mathbb{P}(z\cdot \nabla w)(\tau)d\tau\right\|_{L^\infty(\mathbb{R}^N)}\le C(N,s)\left(t^\alpha+t^{N/4}\right)^{-1}\|z\|_{X^{s,\alpha}_k}\|w\|_{X^{s,\alpha}_k}.
 \end{equation}
\end{lem}
\begin{proof}\hfill\\
Firstly we consider the case $t\le 2$: thanks to Lemma \ref{l.est.HS}
$$ \int_0^te^{\mathbb{P}\Delta(t-\tau)}\mathbb{P}(z\cdot\nabla w)(\tau)d\tau\lesssim \int_0^t(t-\tau)^{-\frac{N}{2q}}\|z\cdot\nabla w(\tau)\|_{L^q(\mathbb{R}^N)}d\tau, $$
for some $q\in(1,\infty)$. We notice that
$$ \|z(\tau)\cdot\nabla w(\tau)\|_{L^q(\mathbb{R}^N)}\le \|z(\tau)\|_{L^\infty(\mathbb{R}^N)}\|\nabla w(\tau)\|_{L^q(\mathbb{R}^N)}\le \tau^{-\alpha}\|z\|_{X^{s,\alpha}_k}\|\nabla w(\tau)\|_{L^q(\mathbb{R}^N)}. $$
We need to ask two conditions over $q$: we need that $\frac{N}{2q}<\frac{1}{2}$ and that $L^q\hookrightarrow H^s$. In this case
$$ \int_0^te^{\mathbb{P}\Delta(t-\tau)}\mathbb{P}(z\cdot\nabla w)(\tau)d\tau\lesssim \|z\|_{X^{s,\alpha}_k}\int_0^t(t-\tau)^{-\frac{N}{2q}}\tau^{-\alpha}\|\nabla w(\tau)\|_{L^q(\mathbb{R}^N)}d\tau\le $$
$$ \le \|z\|_{X^{s,\alpha}_k}\|w\|_{X^{s,\alpha}_k}\left(\int_0^t(t-\tau)^{-\frac{N}{q}}\tau^{-2\alpha}d\tau\right)^\frac{1}{2}\lesssim t^{-\alpha}\|z\|_{X^{s,\alpha}_k}\|w\|_{X^{s,\alpha}_k}, $$
where the last inequality comes from Lemma \ref{l.tec.int.1} choosing $\alpha$ sufficiently near $\frac{1}{2}$. Such a choice of $q$ exists because
$$ \frac{N}{2q}<\frac{1}{2}\:\Leftrightarrow\: q>N $$
and we can find $q>N$ such that $L^q\hookrightarrow H^s$ when $s>\frac{N}{2}-1$ by Sobolev embeddings.

Let us suppose then $t>2$. In this case we can split the integral in three pieces:
$$ \int_0^te^{\mathbb{P}\Delta(t-\tau)}\mathbb{P}(z\cdot \nabla w)(\tau)d\tau= I_1+I_2+I_3, $$
where
$$ I_1\coloneqq \int_0^{t/2}e^{\mathbb{P}\Delta(t-\tau)}\mathbb{P}(z\cdot \nabla w)(\tau)d\tau, $$
$$ I_2\coloneqq \int_{t/2}^{t-1}e^{\mathbb{P}\Delta(t-\tau)}\mathbb{P}(z\cdot \nabla w)(\tau)d\tau, $$
$$ I_3\coloneqq \int_{t-1}^te^{\mathbb{P}\Delta(t-\tau)}\mathbb{P}(z\cdot \nabla w)(\tau)d\tau. $$
We will estimate these three integrals separately: from Lemma \ref{l.est.HS}, choosing $r=2$, we get
$$ \|I_1\|_{L^\infty(\mathbb{R}^N)}\le \int_0^{t/2}(t-\tau)^{-\frac{N}{4}}\|\nabla w (\tau)z(\tau)\|_{L^2(\mathbb{R}^N)}d\tau. $$
When $\tau<\frac{t}{2}$, then $t-\tau\ge \frac{t}{2}$, so
$$ \int_0^{t/2}(t-\tau)^{-\frac{N}{4}}\|\nabla w (\tau)z(\tau)\|_{L^2(\mathbb{R}^N)}d\tau\lesssim t^{-N/4}\int_0^{t/2}\|\nabla w (\tau)z(\tau)\|_{L^2(\mathbb{R}^N)}d\tau. $$
On the other hand
$$ \|\nabla w(\tau)z(\tau)\|_{L^2(\mathbb{R}^N)}\le \|z(\tau)\|_{L^\infty(\mathbb{R}^N)}\|\nabla w(\tau)\|_{L^2(\mathbb{R}^N)}\le (\tau^\alpha+\tau^{N/4})^{-1}\|z\|_{X^{s,\alpha}_k}\|\nabla w(\tau)\|_{L^2(\mathbb{R}^N)}. $$
So
$$ \|I_1\|_{L^\infty(\mathbb{R}^N)} \le t^{-N/4}\|z\|_{X^{s,\alpha}_k}\int_0^{t/2}(\tau^\alpha+\tau^{N/4})^{-1}\|\nabla w(\tau)\|_{L^2(\mathbb{R}^N)}d\tau\le $$
$$ \le t^{-N/4}\|z\|_{X^{s,\alpha}_k}\left(\int_0^\infty (\tau^{2\alpha}+\tau^{N/2})^{-1}d\tau\right)^\frac{1}{2}\|\nabla w\|_{L^2(\mathbb{R}_+;L^2(\mathbb{R}^N))}\lesssim t^{-\frac{N}{4}}\|z\|_{X^{s,\alpha}_k}\|w\|_{X^{s,\alpha}_k}, $$
where we have used that $2\alpha<1$ and $N\ge 3$. In order to estimate $I_2$ we notice as before that
$$ {\rm div }z=0\:\Rightarrow\: z\cdot \nabla w={\rm Div}(z\odot w). $$
Therefore, using Lemma \ref{l.est.HS} with $r=\frac{2N}{N-2}$ and $k=1$, we get
$$ \|I_2\|_{L^\infty(\mathbb{R}^N)}=\left\|\int_{t/2}^{t-1}{\rm Div}\left(e^{\mathbb{P}\Delta(t-\tau)}\mathbb{P}(z\odot w)(\tau)\right)d\tau\right\|_{L^\infty(\mathbb{R}^N)}\le $$
$$ \le \int_{t/2}^{t-1}(t-\tau)^{-\frac{1}{2}-\frac{N-2}{4}}\|z(\tau)w(\tau)\|_{L^\frac{2N}{N-2}(\mathbb{R}^N)}d\tau=\int_{t/2}^{t-1}(t-\tau)^{-\frac{N}{4}}\|z(\tau)w(\tau)\|_{L^\frac{2N}{N-2}(\mathbb{R}^N)}d\tau . $$
Since $\tau<t-1$, then $t-\tau\gtrsim 1+t-\tau$, so
$$ \int_{t/2}^{t-1}(t-\tau)^{-\frac{N}{4}}\|z(\tau)w(\tau)\|_{L^\frac{2N}{N-2}(\mathbb{R}^N)}d\tau\lesssim \int_{t/2}^{t-1}(1+t-\tau)^{-\frac{N}{4}}\|z(\tau)w(\tau)\|_{L^\frac{2N}{N-2}(\mathbb{R}^N)}d\tau. $$
Moreover
$$ \|z(\tau)w(\tau)\|_{L^\frac{2N}{N-2}(\mathbb{R}^N)}\le (\tau^\alpha+\tau^{N/4})^{-1}\|z\|_{X^{s,\alpha}_k}\|w(\tau)\|_{L^\frac{2N}{N-2}(\mathbb{R}^N)}\lesssim $$
$$ \lesssim (\tau^\alpha+\tau^{N/4})^{-1}\|z\|_{X^{s,\alpha}_k}\|\nabla w(\tau)\|_{L^2(\mathbb{R}^N)}, $$
where in the last inequality we have used Sobolev embedding. So
$$ \int_{t/2}^{t-1}(1+t-\tau)^{-\frac{N}{4}}\|z(\tau)w(\tau)\|_{L^\frac{2N}{N-2}(\mathbb{R}^N)}d\tau \lesssim  $$
$$ \lesssim \|z\|_{X^{s,\alpha}_k}\int_{t/2}^{t-1}(1+t-\tau)^{-\frac{N}{4}}(\tau^\alpha+\tau^{N/4})^{-1}\|\nabla w(\tau)\|_{L^2(\mathbb{R}^N)}d\tau\lesssim $$
$$ \lesssim \|z\|_{X^{s,\alpha}_k}\|\nabla w\|_{L^2(\mathbb{R}_+;L^2(\mathbb{R}^N))}\left(\int_0^t (1+t-\tau)^{-\frac{N}{2}}(\tau^{2\alpha}+\tau^{N/2})^{-1}d\tau\right)^\frac{1}{2}\lesssim $$
$$ \lesssim (t^\alpha+t^{N/4})^{-1}\|z\|_{X^{s,\alpha}_k}\|w\|_{X^{s,\alpha}_k}, $$
where in the last inequality we have used Lemma \ref{l.tec.int.2}. For the $I_3$ term, we suppose at the beginning that $s<\frac{N}{2}$. In this case, using Lemma \ref{l.est.HS} with $r=\frac{2N}{N-2s}$, we get
$$ \|I_3\|_{L^\infty(\mathbb{R}^N)}\le \int_{t-1}^t(t-\tau)^{-\frac{N-2s}{4}}\|z(\tau)\nabla w(\tau)\|_{L^{\frac{2N}{N-2s}}(\mathbb{R}^N)}d\tau\le $$
$$ \le \|z\|_{X^{s,\alpha}_k}\int_{t-1}^t(t-\tau)^{-\frac{N-2s}{4}}(\tau^\alpha+\tau^{N/4})^{-1}\|\nabla w(\tau)\|_{L^{\frac{2N}{N-2s}}(\mathbb{R}^N)}d\tau\lesssim
 $$
$$ \le \|z\|_{X^{s,\alpha}_k}\int_{t-1}^t(t-\tau)^{-\frac{N-2s}{4}}(\tau^\alpha+\tau^{N/4})^{-1}\|\nabla  w(\tau)\|_{H^s(\mathbb{R}^N)}d\tau\lesssim $$
$$ \lesssim (t^\alpha+t^{N/4})^{-1}\|z\|_{X^{s,\alpha}_k}\|\nabla w\|_{L^2(\mathbb{R}_+;H^s(\mathbb{R}^N))}\left(\int_{t-1}^t(t-\tau)^{-\frac{N-2s}{2}}d\tau\right)^\frac{1}{2}\simeq $$
$$ \simeq (t^\alpha+t^{N/4})^{-1}\|z\|_{X^{s,\alpha}_k}\|w\|_{X^{s,\alpha}_k}, $$
where in the last inequality we have used the fact that
$$ \frac{N-2s}{2}<1\:\Leftrightarrow\: s>\frac{N}{2}-1. $$
When $s\ge \frac{N}{2}$ it is sufficient to apply Lemma \ref{l.est.HS} with $r>N$ and proceed as before, so the proof is complete.
\end{proof}
\begin{lem}
 Let $k\in\mathbb{N}$ and $s\in\mathbb{R}$ such that $s-k>\frac{N}{2}-1$, let  $z,w\in X^{s,\alpha}_k$ with ${\rm div}z(t)=0$ for a.e. $t>0$, then we can find $\alpha\in\left(0,\frac{1}{2}\right)$ such that for any
$\beta\in\mathbb{N}^N$ with $|\beta|=k$ it holds
 \begin{equation}\label{bil.e.4}
     \left\|D^\beta\int_0^t e^{\mathbb{P}\Delta(t-\tau)}\mathbb{P}(z\cdot \nabla w)(\tau)d\tau\right\|_{L^\infty(\mathbb{R}^N)}\le C(N,s,k)\left(t^\alpha+t^{N/4}\right)^{-1}\|z\|_{X^{s,\alpha}_k}\|w\|_{X^{s,\alpha}_k}.
 \end{equation}
\end{lem}
\begin{proof}\hfill\\
It is sufficient to apply Leibniz formula:
$$ D^\beta \mathbb{P}(z\cdot\nabla w)=\sum_{\gamma\le \beta} \binom{\beta}{\gamma}\mathbb{P}\left(D^\gamma z\cdot \nabla D^{\beta-\gamma}w\right), $$
where
$$ \gamma\le \beta \:\Leftrightarrow\: \gamma_i\le \beta_i\quad \forall i=1,\ldots, N\quad \forall \gamma,\beta\in\mathbb{N}^N. $$
On the other hand, for any $\gamma\le \beta$
$$ D^\gamma z, D^{\beta-\gamma}w\in X^{s-k,\alpha}_0 $$
with $s-k>\frac{N}{2}-1$,
so we can apply Lemma \ref{l.bil.dec.1}:
$$ \left\|D^\beta\int_0^t e^{\mathbb{P}\Delta(t-\tau)}\mathbb{P}(z\cdot \nabla w)(\tau)d\tau\right\|_{L^\infty(\mathbb{R}^N)}\lesssim $$
$$ \lesssim (t^\alpha+t^{N/4})^{-1} \sum_{\gamma\le \beta}\binom{\beta}{\gamma}\|D^\gamma z\|_{X^{s-k,\alpha}_0}\|D^{\beta-\gamma}w\|_{X^{s-k,\alpha}_0}\lesssim (t^\alpha+t^{N/4})^{-1}\|z\|_{X^{s,\alpha}_k}\|w\|_{X^{s,\alpha}_k} $$
\end{proof}
 The next two lemmas will be used to control the $W^{k,\infty}(\mathbb{R}^N;\mathbb{R}^N)$-norm of $\nabla d$:
\begin{lem}\label{l.bil.dec.2}
 Let $k\in\mathbb{N}$, $s>\frac{N}{2}-1$ and $z,w\in X^{s,\alpha}_k$, then we can find $\alpha\in\left(0,\frac{1}{2}\right)$ such that
 $$ \left\|\int_0^t e^{\Delta(t-\tau)}z(\tau)w(\tau)d\tau\right\|_{L^\infty(\mathbb{R}^N)}\le C(N,s)(1+t)^{-N/4}\|z\|_{X^{s,\alpha}_k}\|w\|_{X^{s,\alpha}_k} $$
and
$$ \left\|\nabla \int_0^t e^{\Delta(t-\tau)}z(\tau)w(\tau)d\tau\right\|_{L^\infty(\mathbb{R}^N)}\le C(N,s)\left(t^\alpha+t^{N/4}\right)^{-1}\|z\|_{X^{s,\alpha}_k}\|w\|_{X^{s,\alpha}_k} $$
\end{lem}
\begin{proof}\hfill\\
For what concerns the first estimate, we have already proved the boundness for finite intervals in \eqref{bil.e.2}, so we can suppose $t\ge 2$. We split the integral in two pieces:
$$ \int_0^t e^{\Delta(t-\tau)}z(\tau)w(\tau)d\tau=I_1+I_2,$$
where
$$ I_1=\int_{0}^{t-1}e^{\Delta(t-\tau)}z(\tau)w(\tau)d\tau $$
$$ I_2=\int_{t-1}^t e^{\Delta(t-\tau)}z(\tau)w(\tau)d\tau. $$
Let us estimate these terms separately: by Lemma \ref{l.est.HS}
$$ |I_1|\lesssim \int_0^{t-1}(t-\tau)^{-N/4}\|w(\tau)z(\tau)\|_{L^2(\mathbb{R}^N)}d\tau. $$
Since $\tau<t-1$, then $t-\tau\ge \frac{1}{2} (1+t-\tau)$:
$$ |I_1|\lesssim \|z\|_{X^{s,\alpha}_k}\int_0^t (1+t-\tau)^{-N/4}(\tau^\alpha+\tau^{N/4})^{-1}\|w(\tau)\|_{L^2(\mathbb{R}^N)}d\tau \le $$
$$ \le \|z\|_{X^{s,\alpha}_k}\|w\|_{X^{s,\alpha}_k}\int_0^t(1+t-\tau)^{-N/4} (\tau^{\alpha}+\tau^{N/4})^{-1}d\tau\lesssim t^{-N/4}\|z\|_{X^{s,\alpha}_k}\|w\|_{X^{s,\alpha}_k}, $$
where in the last inequality we have used Lemma \ref{l.tec.int.2}. For what concerns the second term we apply again Lemma \ref{l.est.HS}: let $r\in(0,\infty)$, then
$$ |I_2|\lesssim \int_{t-1}^t(t-\tau)^{-\frac{N}{2r}}\|z(\tau)w(\tau)\|_{L^r(\mathbb{R}^N)}d\tau\le $$
$$ \le \|z\|_{X^{s,\alpha}_k}\int_{t-1}^t(t-\tau)^{-\frac{N}{2r}}(\tau^\alpha+\tau^{N/4})^{-1}\|w(\tau)\|_{L^r(\mathbb{R}^N)}d\tau. $$
If we take $r=N$, we notice that $L^N\hookrightarrow H^s$ for $s>\frac{N}{2}-1$ by Sobolev embedding, so
$$ |I_2|\lesssim \|z\|_{X^{s,\alpha}_k}\|w\|_{X^{s,\alpha}_k}\int_{t-1}^t(t-\tau)^{-1/2}(\tau^\alpha+\tau^{N/4})^{-1}d\tau \lesssim $$
$$ \lesssim t^{-N/4} \|z\|_{X^{s,\alpha}_k}\|w\|_{X^{s,\alpha}_k}\int_{t-1}^t(t-\tau)^{-1/2}d\tau \lesssim t^{-N/4} \|z\|_{X^{s,\alpha}_k}\|w\|_{X^{s,\alpha}_k}, $$
where in the last inequality we have used that $\tau\ge t-1\ge \frac{t}{2}$ for $t\ge 2$. This concludes the first part of the thesis. The other one can be done as in Lemma \ref{l.bil.dec.1}, in fact
$$ \nabla \int_0^t e^{\Delta(t-\tau)}z(\tau)w(\tau)d\tau=I_1+I_2+I_3, $$
where
$$ I_1\coloneqq \nabla \int_0^{t/2}e^{\Delta(t-\tau)}z(\tau)w(\tau)d\tau, $$
$$ I_2\coloneqq \nabla \int_{t/2}^{t-1}e^{\Delta(t-\tau)}z(\tau)w(\tau)d\tau, $$
$$ I_3\coloneqq \nabla \int_{t-1}^te^{\Delta(t-\tau)}z(\tau)w(\tau)d\tau. $$
Since
$$ \nabla (zw)(\tau)=\nabla z(\tau)w(\tau)+z(\tau)\nabla w(\tau), $$
we can repeat the same argument for the proof of Lemma \ref{l.bil.dec.1} and we conclude.
\end{proof}
\begin{lem}
Let $k\in\mathbb{N}$ and $s\in\mathbb{R}$ such that $s-k>\frac{N}{2}-1$, let  $z,w\in X^{s,\alpha}_k$, then we can find $\alpha\in\left(0,\frac{1}{2}\right)$ such that for any $\beta\in\mathbb{N}^N$ with $|\beta|=k+1$ it holds
\begin{equation}\label{bil.e.5}
    \left\|D^\beta\int_0^t e^{\Delta(t-\tau)}z(\tau)w(\tau)d\tau\right\|_{L^\infty(\mathbb{R}^N)}\le C(N,s,k)\left(t^\alpha+t^{N/4}\right)^{-1}\|z\|_{X^{s,\alpha}_k}\|w\|_{X^{a,\alpha}_k}.
\end{equation}
\end{lem}
\begin{proof}\hfill\\
Let $\beta\in\mathbb{N}^N$ with $|\beta|=k+1$, then we can find $i=1,\ldots, N$ such that $\beta=e_i+\widetilde{\beta}$ with $\widetilde{\beta}\in\mathbb{N}^N$. So
$$ D^\beta\int_0^t e^{\Delta(t-\tau)}z(\tau)w(\tau)d\tau=\partial_{x_i}\int_0^t e^{\Delta(t-\tau)}D^{\widetilde{\beta}}(z(\tau)w(\tau))d\tau. $$
If we use Leibniz formula we get that
$$ \left\|\partial_{x_i}\int_0^t e^{\Delta(t-\tau)}D^{\widetilde{\beta}}(z(\tau)w(\tau))d\tau\right\|_{L^\infty(\mathbb{R}^N)}\le  \sum_{\gamma\le \widetilde{\beta}}\binom{\widetilde{\beta}}{\gamma}\left\|\nabla \int_0^t e^{\Delta(t-\tau)}\left(D^\gamma z(\tau)D^{\widetilde{\beta}-\gamma}w(\tau)\right)d\tau\right\|_{L^\infty(\mathbb{R}^N)}. $$
Now we notice that
$$ D^\gamma z,D^{\widetilde{\beta}-\gamma}w\in X^{s-k,\alpha}_0 $$
with $s-k>\frac{N}{2}-1$, so we can apply Lemma \ref{l.bil.dec.1}:
$$ \left\|\partial_{x_i}\int_0^t e^{\Delta(t-\tau)}D^{\widetilde{\beta}}(z(\tau)w(\tau))d\tau\right\|_{L^\infty(\mathbb{R}^N)} \lesssim $$
$$ \lesssim (t^\alpha+t^{N/4})^{-1} \sum_{\gamma\le \widetilde{\beta}} \|D^\gamma z\|_{X^{s-k,\alpha}_0}\|D^{\widetilde{\beta}-\gamma}w\|_{X^{s-k,\alpha}_0}\lesssim (t^\alpha+t^{N/4})^{-1}\|z\|_{X^{s,\alpha}_k}\|w\|_{X^{s,\alpha}_k}. $$
\end{proof}
 We can now prove the global existence and decay result:
\begin{proof}[Proof of Theorem \ref{t.gl-dec.ex.+}]\hfill\\
Since $s>\frac{N}{2}-1$, we can find $k\in\mathbb{N}$ such that $s-k>\frac{N}{2}-1$. As before, the strategy is to use a contraction argument: let
$$ Y_k\coloneqq \left\{(w,\theta)\in X^{s,\alpha}_k\times \Theta^{s,\alpha}_{k}\:\big|\: {\rm div}(w(t))=0\:\:\text{for a.e.}\:\:t>0, \:\:\|(w,\theta)\|_{Y_k}<+\infty\right\}  $$
where
$$ \|(w,\theta)\|_{Y_k}\coloneqq \|w\|_{X^{s,\alpha}_k}+\|\theta\|_{\Theta^{s,\alpha}_{k}} $$
and where $\alpha=\frac{1}{2}-\delta$ for $\delta>0$ sufficiently small. We define as before the map $\phi_k\colon Y_k\to Y_k$ such that $\phi_k(w,\theta)=(u_k,d_k)$ which solves
\begin{equation}
    \left\{\begin{array}{ll}
    (\partial_t-\mathbb{P}\Delta)u_k=-\mathbb{P}(w\cdot \nabla w +{\rm Div}(\nabla \theta\odot\nabla \theta)) & \mathbb{R}_+\times\mathbb{R}^N \\
    (\partial_t-\Delta)d_k=-w\cdot \nabla \theta & \mathbb{R}_+\times\mathbb{R}^N \\
    u_k(0)=u_0,\quad d_k(0)=d_0 & \mathbb{R}^N.
    \end{array}\right.
\end{equation}
Namely,
\begin{equation}\label{Duh.f.gl.}
    \begin{aligned}
        & u_k=e^{\mathbb{P}\Delta t}u_0+\int_0^te^{\mathbb{P}\Delta(t-\tau)}\mathbb{P}f(w,\theta)(\tau)d\tau \\
        & d_k=e^{\Delta t}d_0+\int_0^t e^{\Delta(t-\tau)}g(w,\theta)(\tau)d\tau,
    \end{aligned}
\end{equation}
where
$$ f(w,\theta)=-w\cdot \nabla w-{\rm Div}(\nabla \theta\odot\nabla\theta), $$
$$ g(w,\theta)=w\cdot \nabla \theta. $$
The plan is to prove that $\phi_k$ is a contraction on a suitable subspace of $Y_k$. Firstly, we need to verify the following estimates:
\begin{equation}
    \|\phi_k(w,\theta)\|_{Y_k} \leq C \left( \|u_0\|_{H^s(\mathbb{R}^N)} + \|d_0\|_{L^\infty(\mathbb{R}^N)}+\|\nabla d_0\|_{H^{s}(\mathbb{R}^N)} + \|(w,\theta)\|_{Y_k}^2\right)
\end{equation}
\begin{equation}
    \|\phi_k(w_1,\theta_1)- \phi_k(w_2,\theta_2)\|_{Y_k} \leq  M\left[\|(w_1,\theta_1)\|_{Y_k} + \|(w_2,\theta_2)\|_{Y_k}\right]\|(w_1,\theta_1)-(w_2,\theta_2)\|_{Y_k}
\end{equation}
for some $C,M>0$.

\textbf{Step 1}: Thanks to the divergence free condition for $w$, we have that
$$ w\cdot \nabla w={\rm Div}(w\odot w). $$
So, thanks to Corollary \ref{c.lin.ex.} we have that
$$  \|u_k\|_{X^s}+\|\nabla d_k\|_{X^{s}}\lesssim \|w\odot w\|_{L^2(\mathbb{R}_+;H^s(\mathbb{R}^N))}+\|\nabla \theta\odot\nabla \theta\|_{L^2(\mathbb{R}_+;H^{s}(\mathbb{R}^N))} + \|w\cdot \nabla \theta\|_{L^2(\mathbb{R}_+;H^s(\mathbb{R}^N))}, $$
where $(u_k,d_k)=\phi_k(w,\theta)$. Since $w,\nabla \theta\in X^{s}$ we can apply \eqref{bil.e.3} to get the estimates:
\begin{equation}\label{ww.gl.}
    \|w\odot w\|_{L^2(\mathbb{R}_+;H^s(\mathbb{R}^N))}\lesssim \|w\|_{X^{s}}^2,
\end{equation}
\begin{equation}\label{DtDt.gl.}
    \|\nabla \theta\odot \nabla \theta\|_{L^2(\mathbb{R}_+;H^s(\mathbb{R}^N))}\lesssim  \|\nabla \theta\|_{X^{s}}^2,
\end{equation}
\begin{equation}\label{wDt.gl.}
    \|w\cdot \nabla \theta\|_{L^2(\mathbb{R}_+;H^s(\mathbb{R}^N))}\lesssim \|w\|_{X^{s}}\|\nabla \theta\|_{X^{s}}.
\end{equation}
So, thanks to Corollary \ref{c.lin.ex.} and \eqref{ww.gl.}, \eqref{DtDt.gl.} and \eqref{wDt.gl.} we get
\begin{equation}\label{nl.gl.1}
    \|u_k\|_{X^s}+\|\nabla d_k\|_{X^{s}}\le C\left[\|u_0\|_{H^s(\mathbb{R}^N)}+\|\nabla d_0\|_{H^{s}(\mathbb{R}^N)} + \|(w,\theta)\|_{Y_k}^2\right].
\end{equation}

\textbf{Step 2}: For what concerns the $L^\infty$ estimates, we start from the Duhamel formula \eqref{Duh.f.gl.}: from Lemma \ref{l.decay} we have that for any $m\ge0$ such that $\frac{N}{4}-\frac{m}{2}<\frac{1}{2}$, it holds
$$ \|e^{\mathbb{P}^j\Delta t}w\|_{L^\infty(\mathbb{R}^N)}\lesssim (t^\alpha+t^{N/4})^{-1}\|w\|_{H^m(\mathbb{R}^N)} $$
with $\alpha=\frac{1}{2}-\delta$ and $\delta>0$ sufficiently small. Therefore
\begin{equation}\label{dec.u.1}
 \sup_{t>0}(t^\alpha+t^{N/4})\sum_{|\beta|\le k} \|D^\beta e^{\mathbb{P}\Delta t}u_0\|_{L^\infty(\mathbb{R}^N)}\lesssim \sum_{|\beta|\le k}\|D^\beta u_0\|_{H^{s-k}(\mathbb{R}^N)}\lesssim \|u_0\|_{H^s(\mathbb{R}^N)}
\end{equation}
\begin{equation}\label{Linf.d}
\|e^{\Delta t}d_0\|_{L^\infty(\mathbb{R}^N)}\le \|d_0\|_{L^\infty(\mathbb{R}^N)},
\end{equation}
\begin{equation}\label{dec.d.1}
    \begin{aligned}
        & \sup_{t>0}(t^\alpha+t^{N/4})\sum_{|\beta|\le k}\|D^\beta e^{\Delta t}\nabla d_0\|_{L^\infty(\mathbb{R}^N)} \lesssim \sum_{|\beta|\le k}\|D^\beta \nabla d_0\|_{H^{s-k}(\mathbb{R}^N)}\lesssim \|\nabla d_0\|_{H^{s}(\mathbb{R}^N)},
    \end{aligned}
\end{equation}
where we are using the fact that
$$ \frac{N}{4}-\frac{s-k}{2}<\frac{1}{2}\:\Leftrightarrow\:s-k>\frac{N}{2}-1. $$
On the other hand, we have from the bilinear estimates \eqref{bil.e.3} and \eqref{bil.e.4} and Lemma \ref{l.bil.dec.2} that
\begin{equation}\label{nl.dec.1}
    \sup_{t>0}(t^\alpha+t^{N/4})\sum_{|\beta|\le k}\left\|D^\beta\int_0^t e^{\mathbb{P}\Delta t}\mathbb{P} f(w,\theta)(\tau)d\tau\right\|_{L^\infty(\mathbb{R}^N)} \lesssim \|(w,\theta)\|_{Y_k}^2,
\end{equation}
\begin{equation}\label{nl.dec.2}
    \sup_{t>0}(1+t)^{N/4}\left\|\int_0^t e^{\Delta(t-\tau)}g(w,\theta)(\tau)d\tau\right\|_{L^\infty(\mathbb{R}^N)}\lesssim \|(w,\theta)\|_{Y_k},
\end{equation}
\begin{equation}\label{nl.dec.3}
    \sup_{t>0}(t^\alpha+t^{N/4})\sum_{|\beta|\le k}\left\|D^\beta\nabla \int_0^t e^{\Delta t}D^\beta g(w,\theta)(\tau)d\tau\right\|_{L^\infty(\mathbb{R}^N)} \lesssim \|(w,\theta)\|_{Y_k}^2,
\end{equation}
where we recall
$$ f(w,\theta)=-w\cdot \nabla w-{\rm Div}(\nabla \theta\odot\nabla\theta), $$
$$ g(w,\theta)=w\cdot \nabla \theta. $$
Finally we have then from \eqref{nl.gl.1}, \eqref{dec.u.1}, \eqref{Linf.d}, \eqref{dec.d.1}, \eqref{nl.dec.1}, \eqref{nl.dec.2} and \eqref{nl.dec.3} that exists $C>0$ such that
\begin{equation}\label{nl.gl.2}
    \|\phi_k(w,\theta)\|_{Y_k}\le C\left[\|u_0\|_{H^s(\mathbb{R}^N)}+\|d_0\|_{L^\infty(\mathbb{R}^N)}+\|\nabla d_0\|_{H^{s}(\mathbb{R}^N)}+\|(w,\theta)\|_{Y_k}^2\right].
\end{equation}
With the same approach it can be proven that exists $M>0$ such that
\begin{equation}\label{nl.gl.3}
\begin{aligned}
    & \|\phi_k(w_1,\theta_1)-\phi_k(w_2,\theta_2)\|_{Y_k}\le \\
    & \le M\left(\|(w_1,\theta_1)\|_{Y_k}+\|(w_2,\theta_2)\|_{Y_k}\right)\|(w_1,\theta_1)-(w_2,\theta_2)\|_{Y_k}.
\end{aligned}
\end{equation}

\textbf{Step 3}: Let us define now
$$ Z_\varepsilon^k\coloneqq \{(w,\theta)\in Y_k\mid w(0)=u_0,\quad \theta(0)=d_0,\quad \|(w,\theta)\|_{Y_k}\le 2C\varepsilon\}, $$
where $C>0$ comes from \eqref{nl.gl.2}. We want to prove that $\phi_k$ is a contraction on $Z_\varepsilon^k$. Firstly, we notice that $\phi_k\colon Z_\varepsilon^k\to Z_\varepsilon^k$ for a suitable choice of $\varepsilon$: from \eqref{nl.gl.2} we have
$$ \|\phi_k(w,\theta)\|_{Y_k}\le C\varepsilon\left(1+4C^2\varepsilon\right). $$
Therefore, if we choose $\varepsilon\in(0,1)$ sufficiently small such that
$$ 4C^2\varepsilon\le 1, $$
then $\phi_k(w,\theta)\in Z_\varepsilon^k$. Finally, we notice that $\phi_k\colon Z_\varepsilon^k\to Z_\varepsilon^k$ is a contraction: from \eqref{nl.gl.3} we have
$$ \|\phi_k(w_1,\theta_1)-\phi_k(w_2,\theta_2)\|_{Y_k}\le 4CM\varepsilon \|(w_1,\theta_1)-(w_2,\theta_2)\|_{Y_k}, $$
so for a $\varepsilon>0$ sufficiently small we get that $\phi_k\colon Z_\varepsilon^k\to Z_\varepsilon^k$ is a contraction and therefore we conclude the existence of a solution in $Y_k$.
\end{proof}
\begin{rem}
In the proof we have considered $\alpha\sim \frac{1}{2}$ as a value for the spaces $X^{s,\alpha}_k$ and $\Theta^{s,\alpha}_k$. Actually, the choice of $\alpha$ can be more precise: with a more detailed proof it can be proven that, when $s\in\left(\frac{N}{2}-1,\frac{N}{2}\right]$
$$ \sup_{t>0} (t^\alpha+t^{N/4}) \left[\|u(t)\|_{L^\infty(\mathbb{R}^N)} + \|\nabla d(t)\|_{L^\infty(\mathbb{R}^N)}\right]<+\infty $$
with $\alpha$ as in \eqref{a.value}. Moreover, if $s>\frac{N}{2}$
$$ \sup_{t>0} (1+t)^{N/4} \|u(t)\|_{W^{k,\infty}(\mathbb{R}^N)}+ \sup_{t>0} (t^\alpha+t^{N/4}) \|D^{k+1}u(t)\|_{L^\infty(\mathbb{R}^N)} + $$
$$ + \sup_{t>0} (1+t)^{N/4} \|\nabla d(t)\|_{W^{k,\infty}(\mathbb{R}^N)} + \sup_{t>0} (t^\alpha+t^{N/4}) \| D^{k+1}\nabla d(t)\|_{L^\infty(\mathbb{R}^N)}<+\infty, $$
where $\alpha$ is taken as in \eqref{a.value} with $s-k-1$ in place of $s$
and where $k\in\mathbb{N}$ is defined as
$$ k=\max\left\{h\in\mathbb{N}\:\Bigg|\:  s-h>\frac{N}{2}\right\}. $$
\end{rem}

\section{Stability Results for the Erickson-Leslie system}
\medskip

We finally turn back to the Erickson-Leslie system:
\begin{equation}\label{EL.sys.+}
     \left\{\begin{array}{ll}
       (\partial_t-\Delta)u+\nabla p+u\cdot \nabla u=-{\rm Div}\left(\nabla v\odot\nabla v\right)  & \mathbb{R}_+\times \mathbb{R}^N \\
       {\rm div}u=0 & \mathbb{R}_+\times\mathbb{R}^N \\
       (\partial_t-\Delta)v +u\cdot \nabla v = |\nabla v|^2v & \mathbb{R}_+\times\mathbb{R}^N \\
       |v|=1 & \mathbb{R}_+\times\mathbb{R}^N \\
       u(0)=u_0,\quad v(0)=v_0 & \mathbb{R}^N,
    \end{array}\right.
\end{equation}
Let $\eta,\omega\in \mathcal{S}^{N-1}$ with $\eta\perp\omega$. As we have seen in the introduction, if $v_0(x)\in \text{Span}\{\eta,\omega\}$ for a.e. $x\in\mathbb{R}^N$, i.e.
$$ v_0(x)=\cos d_0(x)\eta+\sin d_0(x) \omega \quad \text{for a.e.}\:x\in\mathbb{R}^N $$
for $d_0\colon\mathbb{R}^N\to[-\pi,\pi]$ and if $(u,p,d)$ solves
\begin{equation}\label{EL.red.sys.+}
    \left\{\begin{array}{ll}
        (\partial_t-\Delta)u+\nabla p+u\cdot \nabla u=-{\rm Div}(\nabla d\otimes\nabla d) & \mathbb{R}_+\times \mathbb{R}^N \\
        {\rm div}u=0 & \mathbb{R}_+\times\mathbb{R}^N \\
        (\partial_t-\Delta)d+u\cdot \nabla d=0 & \mathbb{R}_+\times \mathbb{R}^N \\
        u(0)=u_0,\quad d(0)=d_0 & \mathbb{R}^N
    \end{array}\right.
\end{equation}
then $(u,p,v)$ solves the Ericksen-Leslie system \eqref{EL.sys.+} with
\begin{equation}\label{v.par.}
    v(t,x)=\cos d(t,x)\eta+\sin d(t,x)\omega \quad \text{for a.e.}\:\:(t,x)\in\mathbb{R}_+\times\mathbb{R}^N.
\end{equation}
For this reason, we will assume from now on $v$ to be of the form of \eqref{v.par.}. Before proving local and global existence, we need to understand how the regularity of $d$ and the one of $v$ are connected. Here, the parametrization \eqref{v.par.} makes the things more complicated: let $g(x)\coloneqq \cos(x)\eta+\sin(x)\omega $, then by Faà di Bruno's formula
$$ D^\alpha v(x)=\sum_{\ell=1}^{|\alpha|}g^{(\ell)}(d(x))\sum_{\alpha_1+\ldots+\alpha_\ell=\alpha,\:\:|\alpha_j|\ge 1}D^{\alpha_1}d(x)\cdots D^{\alpha_\ell}d(x) $$
for any $\alpha\in\mathbb{N}^N$ with $|\alpha|\ge 1$. Therefore, it is not easy to see that, if $\nabla d\in H^s(\mathbb{R}^N;\mathbb{R}^N)$ for some $s\in\mathbb{R}$, then also $\nabla v\in H^s(\mathbb{R}^N;\mathbb{R}^{N^2})$. In particular, the case $s\not\in\mathbb{N}$ is hard to study in general. For this reason, we will discuss the case $s\in\mathbb{N}$ in every dimension $N\ge 3$, while we will study the fractional case only for $N=3$ and $s\in\left(\frac{1}{2},1\right)$.

\vspace{3mm}

In order to get the local existence and the global existence for small initial data we will divide the proof in two steps:
\begin{itemize}
    \item Firstly, we need to prove that, when $v_0\in \text{Span}\{\eta,\omega\}$ with $\eta,\omega\in \mathcal{S}^{N-1}$ and $\eta\perp\omega$, then the solution $v(t,x)$ remains on the same plane for a.e. $(t,x)\in\mathbb{R}_+\times\mathbb{R}^N$. Since we have already built a solution $v$ like this by the parametrization \eqref{v.par.}, it is sufficient to prove the uniqueness of the solution for \eqref{EL.sys.+}.
    \item As we have said before, we need to understand the regularity of $v$ in dependence of the function $d$ as in \eqref{v.par.}. In particular, we need to prove that, if $d\in\Theta^s$, then $v\in \Theta^s$ too.
\end{itemize}

\subsection{The case $s\in\mathbb{N}$ and proof of Theorem \ref{t.ex.int.}}

As we have said previously, we want to prove that the system
$$
     \left\{\begin{array}{ll}
       (\partial_t-\mathbb{P}\Delta)u=-\mathbb{P}\left(u\cdot \nabla u-{\rm Div}\left(\nabla v\odot\nabla v\right)\right)  & (0,T)\times \mathbb{R}^N \\
       (\partial_t-\Delta)v +u\cdot \nabla v = |\nabla v|^2v & (0,T)\times\mathbb{R}^N \\
       |v|=1 & (0,T)\times\mathbb{R}^N \\
       u(0)=u_0,\quad v(0)=v_0 & \mathbb{R}^N,
    \end{array}\right.    $$
for $T\in(0,\infty]$ admits a unique solution $(u,v)\in X^s_T\times\Theta^s_T$. The idea will be, again, to use a contraction argument. Therefore, we will need some estimates on the nonlinear part $|\nabla v|^2v$ which was not included in the reduced system \eqref{EL.red.sys.+}.
\begin{lem}
    Let $N\ge 3$, $s\in\mathbb{N}$ with $s>\frac{N}{2}-1$, $T>0$, let $\theta,w,z\in \Theta^s_T$, then
    \begin{equation}\label{stab.es.1}
        \|\nabla z\nabla w \theta\|_{L^2((0,T);H^s(\mathbb{R}^N))}\le C(N,s) T^\gamma\|\theta\|_{\Theta^s_T}\|\nabla z\|_{X^s_T}\|\nabla w\|_{X^s_T}
    \end{equation}
    for some $\gamma=\gamma(N,s)>0$.
\end{lem}
\begin{proof}\hfill\\
The case $s\ge \frac{N}{2}$ is easier, so we consider just the case $s<\frac{N}{2}$. Firstly we notice that
$$ \|\nabla z(\tau)\nabla w(\tau)\theta(\tau)\|_{L^2(\mathbb{R}^N)}\le \|\nabla z(\tau)\nabla w(\tau)\|_{L^2(\mathbb{R}^N)}\|\theta(\tau)\|_{L^\infty(\mathbb{R}^N)}. $$
We can apply then the bilinear estimate \eqref{bil.e.1} in order to get that
$$ \| \nabla z \nabla w\theta\|_{L^2((0,T);L^2(\mathbb{R}^N))}\lesssim T^\gamma \|\theta\|_{L^\infty((0,T);L^\infty(\mathbb{R}^N))}\|\nabla z\|_{X^s_T}\|\nabla w\|_{X^s_T} $$
for some $\gamma>0$. Let us take now $\alpha\in\mathbb{N}^N$ with $|\alpha|=s$, then
$$ D^\alpha(\nabla z\nabla w \theta)=D^\alpha(\nabla z\nabla w)\theta+\sum_{\beta<\alpha}\binom{\alpha}{\beta}D^\beta(\nabla z\nabla w) D^{\alpha-\beta}\theta. $$
For what concerns the first term, we can apply the same argument as before using \eqref{bil.e.1}:
$$ \| D^\alpha(\nabla z \nabla w)\theta\|_{L^2((0,T);L^2(\mathbb{R}^N))}\le \|\theta\|_{L^\infty((0,T);L^\infty(\mathbb{R}^N))}\|D^\alpha(\nabla z\nabla w)\|_{L^2((0,T);L^2(\mathbb{R}^N))}\lesssim $$
$$ \lesssim T^\gamma \|\theta\|_{\Theta^s_T}\|\nabla z\|_{X^s_T}\|\nabla w\|_{X^s_T}. $$
Let us pass to the last term: for any $\beta<\alpha$ we have that
$$ D^{\alpha-\beta}\theta\in L^\infty((0,T);H^{|\beta|+1}(\mathbb{R}^N)), $$
$$ D^\beta(\nabla z\nabla w)\in L^2((0,T);H^{s-|\beta|}(\mathbb{R}^N)), $$
where the last inclusion again comes from the estimate \eqref{bil.e.1}.
$$ \|D^\beta(\nabla z\nabla w)D^{\alpha-\beta}\theta\|_{L^2((0,T);L^2(\mathbb{R}^N))}\le\left\|\|D^\beta(\nabla z\nabla w)\|_{L^p(\mathbb{R}^N)}\|D^{\alpha-\beta}\theta\|_{L^q(\mathbb{R}^N)}\right\|_{L^2((0,T))} $$
with
$$ \frac{1}{2}=\frac{1}{p}+\frac{1}{q}. $$
If we take $p=\frac{2N}{N-2(s-|\beta|)}$ and $q=\frac{2N}{N-2(|\beta|+1)}$, it is sufficient to verify that
$$ \frac{1}{p}+\frac{1}{q}\le \frac{1}{2}, $$
namely
$$ N-2(s-|\beta|)+N-2(|\beta|+1)\le N\:\Leftrightarrow\: s\ge \frac{N}{2}-1. $$
Finally, by Sobolev embedding and \eqref{bil.e.1} we get
$$ \left\|\|D^\beta(\nabla z\nabla w)\|_{L^p(\mathbb{R}^N)}\|D^{\alpha-\beta}\theta\|_{L^q(\mathbb{R}^N)}\right\|_{L^2((0,T))}\lesssim $$
$$ \lesssim \|\nabla \theta\|_{X^s_T}\|\nabla z\nabla w\|_{L^2((0,T);H^s(\mathbb{R}^N))}\lesssim T^\gamma \|\nabla \theta\|_{X^s_T}\|\nabla z\|_{X^s_T}\|\nabla w\|_{X^s_T}, $$
for some $\gamma>0$.
\end{proof}
Let us pass to the $L^\infty$-norms in time and space:
\begin{lem}
 Let $N\ge 3$, $s>\frac{N}{2}-1$, $T>0$, $z,w,\theta\in \Theta^s_T$, then there is $\gamma=\gamma(N,s)>0$ such that
\begin{equation}\label{stab.es.2}
    \left\|\int_0^te^{\Delta(t-\tau)} \nabla z(\tau)\nabla w(\tau)\theta(\tau)d\tau\right\|_{L^\infty((0,T);L^\infty(\mathbb{R}^N))}\le C(N,s)T^\gamma \|\theta\|_{\Theta^s_T}\|\nabla z\|_{X^s_T}\|\nabla w\|_{X^s_T}.
\end{equation}
\end{lem}
To be noticed that in this lemma $s$ may not be an integer.
\begin{proof}\hfill\\
The proof is the same of the one of \eqref{bil.e.2} noticing that
$$ \|\theta\nabla z\nabla w\|_{L^\infty((0,T);L^r(\mathbb{R}^N))}\lesssim \|\theta\|_{\Theta^s_T}\|\nabla z\nabla w\|_{L^\infty((0,T);L^r(\mathbb{R}^N))}, $$
for any $r>\frac{N}{2}$.
\end{proof}
We are now ready to prove the uniqueness result:
\begin{prop}\label{p.un.int.}
    Let $N\ge 3$, $s>\frac{N}{2}-1$, let
    $$ v_0\colon\mathbb{R}^N\to \mathcal{S}^{N-1}, \quad \nabla v_0\in H^s\left(\mathbb{R}^N;\mathbb{R}^{N^2}\right), $$
    let $T\in(0,\infty]$ and let
    $(u_1,v_1)$ and $(u_2,v_2)$ be two solutions  for the system \eqref{EL.sys.+} in $X^s_T\times\Theta^s_T$, then $v_1=v_2$ for a.e. $(t,x)\in(0,T)\times\mathbb{R}^N$.
\end{prop}
\begin{proof}\hfill\\
The proof again is based on a contraction argument: let $T_0\in(0,T]$ and
$$ Y_{T_0}\coloneqq X^s_{T_0}\times\Theta^s_{T_0}. $$
We want to prove that exists $\gamma,C>0$ such that
$$ \|(u_1,v_1)-(u_2,v_2)\|_{Y_{T_0}}\le CT_0^\gamma\left(1+\|(u_1,v_1)\|_{Y_{T}}^2+\|(u_2,v_2)\|_{Y_T}^2\right)\|(u_1,v_1)-(u_2,v_2)\|_{Y_{T_0}}. $$
From Corollary \ref{c.lin.ex.} we have that
$$ \|(u_1,\nabla v_1)-(u_2,\nabla v_2)\|_{X^s_{T_0}}\lesssim \|\left(u_1\odot u_1+\nabla v_1\odot\nabla v_1\right)-\left(u_2\odot u_2-\nabla v_2\odot\nabla v_2\right)\|_{L^2((0,T_0);H^s(\mathbb{R}^N))} +  $$
$$ + \left\| \left(-u_1\cdot \nabla v_1 +|\nabla v_1|^2v_1\right) - \left( -u_2\cdot \nabla v_2 + |\nabla v_2|^2v_2\right)\right\|_{L^2((0,T_0);H^s(\mathbb{R}^N))}. $$
Using \eqref{bil.e.2} and \eqref{stab.es.1} we get that
$$ \|(u_1,\nabla v_1)-(u_2,\nabla v_2)\|_{X^s_{T_0}}\lesssim T_0^\gamma \left(1+\|(u_1,v_1)\|_{Y_T}^2+\|(u_2,v_2)\|_{Y_T}^2\right)\|(u_1,v_1)-(u_2,v_2)\|_{Y_{T_0}}. $$
On the other hand,
$$ v_j=e^{\Delta t}v_0+\int_0^t e^{\Delta(t-\tau)}g(u_j,\nabla v_j)(\tau)d\tau \quad j=1,2, $$
where
$$ g(u_j,\nabla v_j)\coloneqq -u_j\cdot \nabla v_j +|\nabla v_j|^2v_j\quad j=1,2. $$
So
$$ \|v_1-v_2\|_{L^\infty((0,T_0);L^\infty(\mathbb{R}^N))}= \left\|\int_0^t e^{\Delta(t-\tau)}g(u_j,\nabla v_j)(\tau)d\tau\right\|_{L^\infty((0,T_0);L^\infty(\mathbb{R}^N))}\lesssim $$
$$ \lesssim T_0^\gamma\left(1+\|(u_1,v_1)\|_{Y_T}^2+\|(u_2,v_2)\|_{Y_T}^2\right)\|(u_1,v_1)-(u_2,v_2)\|_{Y_{T_0}},  $$
where in the last inequality we have used \eqref{bil.e.2} and \eqref{stab.es.2}. Finally we have that
$$ \|(u_1,v_1)-(u_2,v_2)\|_{Y_{T_0}}\le C T_0^\gamma \left(1+\|(u_1,v_1)\|_{Y_T}^2+\|(u_2,v_2)\|_{Y_T}^2\right)\|(u_1,v_1)-(u_2,v_2)\|_{Y_{T_0}}.$$
for some $\gamma,C>0$. Therefore, for $T_0=T_0(\|(u_1,v_1)\|_{Y_{T}},\|(u_2,v_2)\|_{Y_{T}})\ll1$ we get that
$$ \|(u_1,v_1)-(u_2,v_2)\|_{Y_{T_0}}\le \frac{1}{2}\|(u_1,v_1)-(u_2,v_2)\|_{Y_{T_0}}, $$
so $u_1=u_2$ and $v_1=v_2$ a.e. on $(0,T_0)\times\mathbb{R}^N$. On the other hand, $T_0$ depends only on the norm in all $(0,T)$ of $v_1$ and $v_2$, so it is easy to see that the extension in $(0,T)$ of the solutions $v_1$ and $v_2$ is unique.
\end{proof}
Let us pass to the regularity of $v$ in dependence of $d$, for
$$ v(t,x)=\cos d(t,x) \eta+\sin d(t,x) \omega\quad \text{for a.e.}\:\:(t,x)\in\mathbb{R}_+\times\mathbb{R}^N. $$
Since $d$ has to be the solution of the reduced system \eqref{EL.red.sys.+}, we want to prove that, when $d\in \Theta^s_T$, then also $v\in \Theta^s_T$ for every $T\in(0,\infty]$. Firstly, we state some technical lemma we will use in this section, whose proofs can be found in the Appendix:
\begin{lem}\label{l.reg.1}
Let $N\ge  3$, $\ell\in\mathbb{N}$ with $\ell\ge 1$, let $s_1,\ldots,s_\ell\in\mathbb{N}$ and $s\in\mathbb{N}$ with $s> \frac{N}{2}-1$ such that
$$ s_1+\ldots+s_\ell\le s-(\ell-1), $$
let $w\in H^s(\mathbb{R}^N)$, then
$$ \|\nabla^{s_1}w\cdots\nabla^{s_\ell}w\|_{L^2(\mathbb{R}^N)}\lesssim \|w\|_{H^s(\mathbb{R}^N)}^\ell. $$
\end{lem}
\begin{lem}\label{l.reg.2}
Let $N\ge  3$, $\ell\in\mathbb{N}$ with $\ell\ge 2$, let $s_1,\ldots,s_\ell\in\mathbb{N}$ and $s\in\mathbb{N}$ with $s>\frac{N}{2}-1$ such that
$$ s_1+\ldots+s_\ell\le s-(\ell-2), $$
let $w\in H^{s+1}(\mathbb{R}^N)$, then
$$ \left\|\nabla^{s_1}w\cdots\nabla^{s_\ell}w\right\|_{L^2(\mathbb{R}^N)}\lesssim \|\nabla w\|_{H^s(\mathbb{R}^N)}\|w\|_{H^s(\mathbb{R}^N)}^{\ell-1} $$
\end{lem}
When $s\in\mathbb{N}$ the behaviour of the $H^s$-norm in space is clear:
\begin{lem}\label{l.Hk.eq.}
Let $N\ge 3$, $s\in\mathbb{N}$ with $s>\frac{N}{2}-1$, let $\eta,\omega\in \mathcal{S}^{N-1}$ with $\eta\perp\omega$, let
$$ v(x)=\cos d(x)\eta+\sin d(x)\omega $$
for some $d\in L^\infty(\mathbb{R}^N)$, then $\nabla v\in H^s(\mathbb{R}^N;\mathbb{R}^{N^2})$ if and only if $\nabla d\in H^s(\mathbb{R}^N;\mathbb{R }^N)$, with
$$ \|\nabla v\|_{H^s(\mathbb{R}^N)}\le C(N,s) \|\nabla d\|_{H^s(\mathbb{R}^N)}\left(1+\|\nabla d\|_{H^s(\mathbb{R}^N)}^{s}\right), $$
$$ \|\nabla d\|_{H^s(\mathbb{R}^N)}\le C(N,s) \|\nabla v\|_{H^s(\mathbb{R}^N)}\left(1+\|\nabla v\|_{H^s(\mathbb{R}^N)}^{s}\right). $$
Moreover, if $\nabla d\in H^{s+1}(\mathbb{R}^N;\mathbb{R}^{N^2})$, then $\nabla v\in H^{s+1}(\mathbb{R}^N;\mathbb{R}^{N^2})$ with
 $$ \|\nabla^2 v\|_{H^s(\mathbb{R}^N)}\le C(N,s) \|\nabla^2 d\|_{H^{s}(\mathbb{R}^N)}\left(1+\|\nabla d\|_{H^s(\mathbb{R}^N)}^{s+1}\right). $$
\end{lem}
\begin{proof}\hfill\\
For what concerns the estimates of $\|\nabla v\|_{H^s}$, we recall that $|\nabla v|=|\nabla d|$, so $\|\nabla v\|_{L^2}=\|\nabla d\|_{L^2}$. On the other hand, by Faà di Bruno's formula, we have that for any $\alpha\in\mathbb{N}^N$ with $|\alpha|=s+1$
$$ D^\alpha v=\sum_{\ell=1}^{|\alpha|}\left(f_1^{(\ell)}(d)\eta+f_2^{(\ell)}(d)\omega\right)\sum_{\alpha_1+\ldots+\alpha_\ell=\alpha,\:\:|\alpha_j|\ge 1}D^{\alpha_1}d\cdots D^{\alpha_\ell}d, $$
where $f_1=\cos$ and $f_2=\sin$. In particular
$$ \|D^\alpha v\|_{L^2(\mathbb{R}^N)}\lesssim \|D^\alpha d\|_{L^2(\mathbb{R}^N)}+\sum_{\ell=2}^{|\alpha|}\sum_{\alpha_1+\ldots+\alpha_\ell=\alpha,\:\:|\alpha_j|\ge 1}\|D^{\alpha_1}d\cdots D^{\alpha_\ell}d\|_{L^2(\mathbb{R}^N)}. $$
Now we notice that
$$ |D^{\alpha_1}d\cdots D^{\alpha_\ell}d|\le|\nabla^{s_1}\nabla d\cdots\nabla^{s_\ell}\nabla d|, $$
with
$$ s_1+\cdots+s_\ell=\sum_{j=1}^\ell|\alpha_j|-\ell=|\alpha|-\ell=s+1-\ell. $$
So, using Lemma \ref{l.reg.1} we get that
$$ \|D^\alpha v\|_{L^2(\mathbb{R}^N)}\lesssim \sum_{\ell=1}^{|\alpha|}\|\nabla d\|_{H^s(\mathbb{R}^N)}^\ell, $$
which gives us the first estimate.

\vspace{2mm}

In order to get the second estimate, we need to prove the following fact:
$$ |\nabla^k\nabla d|\lesssim  \sum_{j=1}^{k+1}\sum_{h_1+\ldots+h_j\le k-(j-1)}|\nabla^{h_1}\nabla v|\cdots|\nabla^{h_j}\nabla v| \quad \forall k\in\mathbb{N}. $$
We will prove it by induction over $k$: if $k=0$ we have already seen that $|\nabla v|=|\nabla d|$ so the estimate is clear. On the other hand, by the Faà di Bruno's formula we have used before, for any $|\alpha|=k+1$ it holds
$$ |D^\alpha v|^2=|D^\alpha d|^2\sin^2d-\sin dD^\alpha d\cdot \sum_{\ell=2}^{|\alpha|}f_1^{(\ell)}(d)\mathcal{D}_\ell +\left|\sum_{\ell=2}^{|\alpha|}f_1^{(\ell)}(d)\mathcal{D}_\ell\right|^2+ $$
$$ + |D^\alpha d|^2\cos^2d+\cos dD^\alpha d\cdot \sum_{\ell=2}^{|\alpha|}f_2^{(\ell)}(d)\mathcal{D}_\ell +\left|\sum_{\ell=2}^{|\alpha|}f_2^{(\ell)}(d)\mathcal{D}_\ell\right|^2,  $$
where
$$ \mathcal{D}_\ell\coloneqq \sum_{\alpha_1+\ldots+\alpha_\ell=\alpha,\:\:|\alpha_j|\ge 1} D^{\alpha_1}d\cdots D^{\alpha_\ell}d. $$
In particular
$$    |D^\alpha v|^2=|D^\alpha d|^2+(\cos d -\sin d )D^\alpha d \cdot \sum_{\ell=2}^{|\alpha|}\left(f_1^{(\ell)}(d)+f_2^{(\ell)}(d)\right)\mathcal{D}_\ell+ \sum_{j=1}^2\left|\sum_{\ell=2}f_j^{(\ell)}(d)\mathcal{D}_\ell\right|^2. $$
Therefore we get as before that
$$ |\nabla^k\nabla d|\lesssim |\nabla^k\nabla v|+\sum_{\ell=2}^{k+1}\sum_{k_1+\cdots+k_\ell=k-(\ell-1)}|\nabla^{k_1}\nabla d|\cdots|\nabla^{k_\ell}\nabla d|. $$
Now we use the induction: since $\ell\ge 2$, $k_j<k$ for any $j=1,\ldots,\ell$ so
$$ \sum_{\ell=2}^{k+1}\sum_{k_1+\cdots+k_\ell=k-(\ell-1)}|\nabla^{k_1}\nabla v|\cdots|\nabla^{k_\ell}\nabla v|\lesssim $$
$$ \lesssim \sum_{\ell=2}^{k+1}\prod_{i=1}^\ell \left(\sum_{j_i=1}^{k_i+1}\sum_{h_1^i+\cdots+h_{j_i}^i\le k_i-(j_i-1)}|\nabla^{h_1^i}\nabla v|\cdots|\nabla^{h_{j_i}^i}\nabla v|\right). $$
A generic element of the previous term is the following:
$$ \prod_{i=1}^\ell |\nabla^{h_1^i}\nabla v|\cdots|\nabla^{h_{j_i}^i}\nabla v|. $$
If we call $\ell_{tot}=j_1+\cdots+j_\ell$, then
$$ \left(h^1_1+\cdots+h^1_{j_1}\right)+\cdots+\left(h^\ell_1+\cdots+h^\ell_{j_\ell}\right)\le \sum_{i=1}^\ell [k_i-(j_i-1)]= $$
$$ = \sum_{i=1}^\ell k_i-\ell_{tot}+\ell= k-(\ell-1)-\ell_{tot}+\ell=k-(\ell_{tot}-1). $$
Therefore, we have the thesis of the induction:
$$ |\nabla^k\nabla d|\lesssim |\nabla^k\nabla v|+\sum_{\ell=2}^{k+1}\sum_{k_1+\cdots+k_\ell\le k-(\ell-1)}|\nabla^{k_1}\nabla v|\cdots|\nabla^{k_\ell}\nabla v|\le $$
$$ \le \sum_{\ell=1}^{k+1}\sum_{k_1+\cdots+k_\ell\le k-(\ell-1)}|\nabla^{k_1}\nabla v|\cdots|\nabla^{k_\ell}\nabla v|. $$
Finally, as before, we get from Lemma \ref{l.reg.1} that
$$ \|\nabla d\|_{H^s(\mathbb{R}^N)}\lesssim \sum_{\ell=1}^{s+1}\|\nabla v\|_{H^s(\mathbb{R}^N)}^\ell. $$
The estimate for $\nabla^2 v$ is similar: we need to prove that
    $$ \|\nabla^2v\|_{L^2(\mathbb{R}^N)}+\|\nabla^{s+2}v\|_{H^k(\mathbb{R}^N)}\lesssim \|\nabla^2 d\|_{H^{s}(\mathbb{R}^N)}\left(1+\|\nabla d\|_{H^s(\mathbb{R}^N)}^{s+1}\right). $$
    The first estimate is easy:
    $$ \|\nabla^2 v\|_{L^2(\mathbb{R}^N)}\lesssim \|\nabla^2d\|_{L^2(\mathbb{R}^N)}+\|\nabla d\nabla d\|_{L^2(\mathbb{R}^N)}. $$
    and
    $$ \|\nabla d\nabla d\|_{L^2(\mathbb{R}^N)}\le \|\nabla d\|_{L^N(\mathbb{R}^N)}\|\nabla d\|_{L^{\frac{2N}{N-2}}(\mathbb{R}^N)}\lesssim \|\nabla d\|_{H^s(\mathbb{R}^N)}\|\nabla^2d\|_{L^2(\mathbb{R}^N)}, $$
    where in the last inequality we have used Sobolev inequality combined with the hypothesis $s>\frac{N}{2}-1$. Let us pass to the second term: let $\alpha\in\mathbb{N}^N$ with $|\alpha|=s+2$, then we have as before that
    $$ \|\nabla^\alpha v\|_{L^2(\mathbb{R}^N)}\lesssim \sum_{\ell=1}^{|\alpha|} \sum_{\alpha_1+\ldots\alpha_\ell=\alpha,\:|\alpha_j|\ge 1}\|D^{\alpha_1}d\cdots D^{\alpha_\ell}d\|_{L^2(\mathbb{R}^N)}. $$
    This time
    $$ |D^{\alpha_1}d\cdots D^{\alpha_\ell}d|\le |\nabla^{s_1}\nabla d\cdots \nabla^{s_\ell}\nabla d|, $$
    with
    $$ s_1+\ldots s_\ell=|\alpha_1|+\cdots+|\alpha_\ell|-\ell=|\alpha|-\ell=s+2-\ell. $$
    Then, using Lemma \ref{l.reg.2} we conclude.
\end{proof}
Thanks to the previous lemma we get the regularity result we were looking for:
\begin{prop}\label{p.reg.int.}
    Let $N\ge 3$, $T\in(0,\infty]$, $s\in\mathbb{N}$ with $s>\frac{N}{2}-1$, let $\eta,\omega\in \mathcal{S}^{N-1}$ with $\eta\perp\omega$, let $v\colon\mathbb{R}^N\to \mathcal{S}^{N-1}$ with
    $$ v(t,x)=\cos d(t,x)\eta+\sin d(t,x) \omega\quad \text{for a.e.}\:\:(t,x)\in(0,T)\times \mathbb{R}^N, $$
    with $d\in \Theta^s_T$, then $v\in\Theta^s_T$.
\end{prop}
\begin{proof}\hfill\\
    From Lemma \ref{l.Hk.eq.} we have that
    $$ \|\nabla v\|_{L^\infty((0,T);H^s(\mathbb{R}^N))}\lesssim \left\|\|\nabla d\|_{H^s(\mathbb{R}^N)}\left(1+\|\nabla d\|_{H^s(\mathbb{R}^N)}^{s}\right) \right\|_{L^\infty((0,T))}\lesssim $$
    $$ \lesssim \|\nabla d\|_{L^\infty((0,T);H^s(\mathbb{R}^N))}\left(1+\|\nabla d\|_{L^\infty((0,T);H^s(\mathbb{R}^N))}^{s}\right), $$
    $$ \|\nabla^2 v\|_{L^2((0,T);H^s(\mathbb{R}^N))}\lesssim \left\| \|\nabla^2 d\|_{H^{s}(\mathbb{R}^N)}\left(1+\|\nabla d\|_{H^s(\mathbb{R}^N)}^{s+1}\right)\right\|_{L^2((0,T))}\lesssim $$
    $$ \lesssim \|\nabla^2 d\|_{L^2((0,T);H^{s}(\mathbb{R}^N))}\left(1+\|\nabla d\|_{L^\infty((0,T);H^s(\mathbb{R}^N))}^{s+1}\right). $$
\end{proof}
Now we are ready to prove local and global existence for the case $s\in\mathbb{N}$:
\begin{proof}[Proof of Theorem \ref{t.ex.int.}, Local existence]\hfill\\
Since $v_0\in\text{Span}\{\eta,\omega\}$, we can find a function $d_0\colon\mathbb{R}^N\to[-\pi,\pi]$ such that
$$ v_0(x)=\cos d_0(x)\eta+\sin d_0(x)\omega\quad \text{for a.e.}\:\:x\in\mathbb{R}^N. $$
Since $\nabla v_0\in H^s(\mathbb{R}^N;\mathbb{R}^{N^2})$ we have from Lemma \ref{l.Hk.eq.} that $\nabla d_0\in H^s(\mathbb{R}^N;\mathbb{R}^N)$ too. Moreover
$$ |d_0|^2\lesssim |1-\cos(d_0)|\le |1-\cos(d_0)|+|\sin(d_0)|=|v_0-\eta|, $$
so we can find $T=T(R)>0$ and $(u,p,d)$ which solves the reduced system \eqref{EL.red.sys.+} with $(u,d)\in X^s_T\times \Theta^s_T$. If we define
$$ v(t,x)=\cos d(t,x) \eta+ \sin d(t,x)\omega \quad \text{for a.e.}\:\:(t,x)\in(0,T)\times\mathbb{R}^N, $$
then, by construction, $(u,p,v)$ solves the Ericksen-Leslie system \eqref{EL.sys.+}. Moreover, $v\in\Theta_T^s$ by Proposition \ref{p.reg.int.} and the choice of $(u,v)$ is unique by Proposition \ref{p.un.int.}, so we conclude.
\end{proof}
\begin{proof}[Proof of Theorem \ref{t.ex.int.}, Global existence and decay]\hfill\\
As before, we can find $d_0\colon\mathbb{R}^N\to[-\pi,\pi]$ such that
$$ v_0(x)=\cos d_0(x)\eta+\sin d_0(x)\omega\quad\text{for a.e.}\:\:x\in\mathbb{R}^N, $$
with $d_0\in L^\infty(\mathbb{R}^N)$,  $\nabla d_0\in H^s(\mathbb{R}^N;\mathbb{R}^N)$ with
$$ \|d_0\|_{L^\infty(\mathbb{R}^N)}^2+\|\nabla d_0\|_{H^s(\mathbb{R}^N)}\lesssim \varepsilon. $$
So, choosing $\varepsilon$ sufficiently small we get the existence of a solution $(u,p,d)$ which solves the reduced system \eqref{EL.red.sys.+}. Uniqueness and regularity follows by Proposition \ref{p.un.int.} and Proposition \ref{p.reg.int.}. In particular we have that
$$ \|u\|_{X^s}+\|\nabla p\|_{L^2((0,T);H^s(\mathbb{R}^N))}+\|\nabla v\|_{X^s}\lesssim \varepsilon. $$
Moreover,
$$ |v-\eta|=|1-\cos (d)|+|\sin (d)|\lesssim |d|^2+|d|\lesssim \varepsilon^2+\varepsilon\lesssim \varepsilon $$
for $\varepsilon<1$. The decay estimates for $u$ and $\nabla d$ again follows from Lemma \ref{l.Hk.eq.}.
\end{proof}

\subsection{The case $s\not\in\mathbb{N}$ and proof of Theorem \ref{t.ex.fr.}}

As before, we have to prove firstly a uniqueness result for the system \eqref{EL.sys.+}. For this reason, we need a bilinear estimate as the one of \eqref{stab.es.1}:
\begin{lem}
    Let $s\in\left(\frac{1}{2},1\right)$, $T>0$, let $\theta,w,z\in \Theta^s_T$, then
    \begin{equation}\label{stab.es.3}
        \|\nabla z\nabla w \theta\|_{L^2((0,T);H^s(\mathbb{R}^3))}\le C_s T^\gamma\|\theta\|_{\Theta^s_T}\|\nabla z\|_{X^s_T}\|\nabla w\|_{X^s_T}
    \end{equation}
    for some $\gamma=\gamma_s>0$.
\end{lem}
\begin{proof}\hfill\\
    We already know from \eqref{stab.es.1} that
    $$ \|\nabla z\nabla w \theta\|_{L^2((0,T);L^2(\mathbb{R}^3))}\lesssim  T^\gamma\|\theta\|_{\Theta^s_T}\|\nabla z\|_{X^s_T}\|\nabla w\|_{X^s_T} $$
    for some $\gamma>0$, so we focus on the seminorm
    $$ \left\|\left\|\frac{(\nabla z\nabla w \theta)(\tau,x+h)-(\nabla z\nabla w \theta)(\tau, x)}{|h|^s}\right\|_{L^2\left(\frac{dh}{|h|^3}\right)}\right\|_{L^2(\mathbb{R}^3)}. $$
    It is easy to see that
    $$ \left\|\left\|\frac{(\nabla z\nabla w \theta)(\tau, x+h)-(\nabla z\nabla w \theta)(\tau, x)}{|h|^s}\right\|_{L^2\left(\frac{dh}{|h|^3}\right)}\right\|_{L^2(\mathbb{R}^3)}\lesssim I_1+I_2+I_3, $$
    where

    $$ I_1(\tau)=\left\|\left\|\frac{\nabla w(\tau, x+h) \theta(\tau, x+h)(\nabla z(\tau, x+h)-\nabla z(\tau, x))}{|h|^s}\right\|_{L^2\left(\frac{dh}{|h|^3}\right)}\right\|_{L^2(\mathbb{R}^3)}  $$
    $$ I_2(\tau)= \left\|\left\|\frac{\nabla z(\tau, x)\theta(\tau, x+h)(\nabla w(\tau, x+h)-\nabla w(\tau, x))}{|h|^s}\right\|_{L^2\left(\frac{dh}{|h|^3}\right)}\right\|_{L^2(\mathbb{R}^3)} $$
    $$ I_3(\tau)= \left\|\left\|\frac{\nabla z(\tau, x)\nabla w(\tau, x) (\theta(\tau, x+h)-\theta(\tau, x))}{|h|^s}\right\|_{L^2\left(\frac{dh}{|h|^3}\right)}\right\|_{L^2(\mathbb{R}^3)}. $$
    For what concerns the first term
    $$ I_1(\tau)^2\le \|\theta(\tau)\|_{L^\infty(\mathbb{R}^3)}^2\int_{\mathbb{R}^3}\left(\int_{\mathbb{R}^3}|\nabla w(\tau,x+h)|^2|\nabla z(\tau,x+h)-\nabla z(\tau,x)|^2dx\right)\frac{dh}{|h|^{2s+3}}. $$
    Now we notice that
    $$ \int_{\mathbb{R}^3}\left(\int_{\mathbb{R}^3}|\nabla w(\tau,x+h)|^2|\nabla z(\tau,x+h)-\nabla z(\tau,x)|^2dx\right)\frac{dh}{|h|^{2s+3}} = $$
    $$ = \int_{\mathbb{R}^3}|\nabla w(\tau,x)|^2\left(\int_{\mathbb{R}^3}\frac{|\nabla z(\tau,x)-\nabla z(\tau,x-h)|^2}{|h|^{2s+3}}dh\right)dx\le $$
    $$ \le \|\nabla w(\tau)\|_{L^p(\mathbb{R}^3)}^2\left\|\left\|\nabla z(x+h)-\nabla z(x)\right\|_{L^2\left(\frac{dh}{|h|^3}\right)}\right\|_{L^q(\mathbb{R}^3)}^2\le   \|\nabla w(\tau)\|_{L^p(\mathbb{R}^3)}^2\|\nabla z(\tau)\|_{H^s_q(\mathbb{R}^3)}^2,  $$
    for
    $$ \frac{1}{2}=\frac{1}{p}+\frac{1}{q}. $$
    If we take $p=\frac{6}{3-2s}$, then $q=\frac{3}{s}\in\left[2,6\right)$ for $s\in\left(\frac{1}{2},1\right)$ so by Sobolev embedding we can find $\gamma\in(0,1)$ such that
    $$ \left\|I_1\right\|_{L^2((0,T))} \lesssim \|\theta\|_{\Theta^s_T}\|\nabla z\|_{L^\infty((0,T);H^s(\mathbb{R}^N))}\|\nabla w\|^\gamma_{L^\infty((0,T);H^s(\mathbb{R}^3))}\|\nabla^2w\|_{L^{2(1-\gamma)}((0,T);H^s(\mathbb{R}^3))} \lesssim $$
    $$ \lesssim T^\gamma \|\theta\|_{\Theta^s_T}\|\nabla z\|_{X^s_T}\|\nabla w\|_{X^s_T}. $$
    The second term is analogous so we consider just the third one:
    $$ I_3(\tau)^2= \int_{\mathbb{R}^3}\left(\int_{\mathbb{R}^3}|\nabla z(\tau,x)|^2|\nabla w(\tau,x)|^2|\theta(\tau,x+h)-\theta(\tau,x)|^2dx\right)\frac{dh}{|h|^{2s+3}}\le $$
    $$ \le \|\nabla z(\tau)\nabla w(\tau)\|_{L^3(\mathbb{R}^3) }^2\int_{\mathbb{R}^3}\|\theta(\tau,x+h)-\theta(\tau,x)\|_{L^6(\mathbb{R}^3)}^2\frac{dh}{|h|^{2s+3}}.  $$
    Since $\theta\in \widehat{H}^1(\mathbb{R}^3)$, we know from Remark \ref{rem.Hom.Sob.} that there exists $c\in\mathbb{R}$ constant such that
    $$ \|\theta(\tau)-c\|_{L^6(\mathbb{R}^3)}\lesssim \|\nabla \theta(\tau)\|_{L^2(\mathbb{R}^3)}\quad \text{for a.e.}\:\:\tau\in(0,T). $$
    Therefore,
    $$ \|\theta(\tau,x+h)-\theta(\tau,x)\|_{L^6(\mathbb{R}^3)}\lesssim \|\nabla \theta(\tau, x+h)-\nabla \theta(\tau, x)\|_{L^2(\mathbb{R}^3)} $$
    and
    $$ \int_{\mathbb{R}^3}\|\theta(\tau,x+h)-\theta(\tau,x)\|_{L^6(\mathbb{R}^3)}^2\frac{dh}{|h|^{2s+3}}\lesssim \|\nabla \theta(\tau)\|_{H^s(\mathbb{R}^3)}^2\le \|\nabla \theta\|_{X^s_T}^2. $$
    Finally, by Sobolev embedding and the estimate \eqref{bil.e.1} we get
    $$ \left\|I_3\right\|_{L^2((0,T))}\lesssim \|\nabla \theta\|_{X^s_T}\|\nabla z\nabla w\|_{L^2((0,T);H^s(\mathbb{R}^3))}\lesssim T^\gamma\|\nabla \theta\|_{X^s_T}\|\nabla z\|_{X^s_T}\|\nabla w\|_{X^s_T} $$
    for some $\gamma>0$, so we conclude.
\end{proof}
As before, applying \eqref{stab.es.2} and \eqref{stab.es.3} it can be proven the uniqueness result:
\begin{prop}\label{p.un.fr}
    Let $s\in\left(\frac{1}{2},1\right)$, let
    $$ v_0\colon\mathbb{R}^3\to \mathcal{S}^2,\quad \nabla v_0\in H^s\left(\mathbb{R}^3;\mathbb{R}^{9}\right), $$
    let $T\in(0,\infty]$ and let
    $(u_1,v_1)$ and $(u_2,v_2)$ be two solutions  for the system \eqref{EL.sys.+} in $X^s_T\times\Theta^s_T$, then $v_1=v_2$ for a.e. $(t,x)\in(0,T)\times\mathbb{R}^3$.
\end{prop}
Let us pass to the regularity result. Here the main problem is that we are not able to prove the equivalence of the $H^s$-norms of $\nabla v$ and $\nabla d$ in the sense of Lemma \ref{l.Hk.eq.}. As we will see, we have just one inequality.
\begin{lem}\label{l.Hs.es.1}
    Let $N\ge 3$, let $d\colon\mathbb{R}^N\to\mathbb{R}$, $f\colon \mathbb{R}^N\to\mathbb{R}^N$, $g\in C^1(\mathbb{R})$ and $w\colon\mathbb{R}^N\to\mathbb{R}^N$ defined as
    $$ w(x)=f(x)g(d(x))\quad \text{a.e.}\:\:x\in\mathbb{R}^N, $$
    then for any $s\in[0,1)$, if $f\in H^s(\mathbb{R}^N;\mathbb{R}^N)$, $g,g^\prime\in L^\infty(\mathbb{R})$ and $\nabla d\in H^s(\mathbb{R}^N;\mathbb{R}^N)$, then $w\in H^s(\mathbb{R}^N;\mathbb{R}^N)$ with
    $$ \|w\|_{H^s(\mathbb{R}^N)}\le C(N,s)\left[\|g\circ d\|_{L^\infty(\mathbb{R}^N)}\|f\|_{H^s(\mathbb{R}^N)} + \|g^\prime\|_{L^\infty(\mathbb{R})}\|f\|_{L^N(\mathbb{R}^N)}\|\nabla d\|_{H^s(\mathbb{R}^N)}\right]. $$
\end{lem}
\begin{proof}\hfill\\
    Firstly, we notice that
    $$ \|w\|_{L^2(\mathbb{R}^N)}\le \|g\circ d\|_{L^\infty(\mathbb{R}^N)}\|f\|_{L^2(\mathbb{R}^N)}. $$
    For what concerns the seminorm of $H^s$, we can see as in the proof of the estimate \eqref{stab.es.3} that
    $$ \left\|\left\|\frac{f(x+h)g(d(x+h))-f(x)g(d(x))}{|h|^s}\right\|_{L^2\left(\frac{dh}{|h|^N}\right)}\right\|_{L^2(\mathbb{R}^N)}^2\lesssim $$
    $$ \lesssim  \|g\circ d\|_{L^\infty(\mathbb{R}^N)}^2\|f\|_{H^s(\mathbb{R}^N)}^2+\|f\|_{L^N(\mathbb{R}^N)}^2\int_{\mathbb{R}^N}\|g(d(x+h))-g(d(x))\|_{L^\frac{2N}{N-2}(\mathbb{R}^N)}^2\frac{dh}{|h|^{2s+N}}. $$
    For the regularity of $g$ we have that
    $$ \|g(d(x+h))-g(d(x))\|_{L^\frac{2N}{N-2}(\mathbb{R}^N)}\le \|g^\prime\|_{L^\infty(\mathbb{R})}\|d(x+h)-d(x)\|_{L^\frac{2N}{N-2}(\mathbb{R}^N)},  $$
    so we can conclude using again Remark \ref{rem.Hom.Sob.}.
\end{proof}
\begin{lem}\label{l.Hs.es.2}
Let $s\in\left(\frac{1}{2},1\right)$, $\eta,\omega\in S^{2}$ with $\eta\perp\omega$, let
$$ v(x)=\cos d(x)\eta+\sin d(x)\omega\quad \text{for a.e.}\:\:x\in\mathbb{R}^3 $$
for some $d\in L^\infty(\mathbb{R}^3)$ with $\nabla d\in H^s(\mathbb{R}^3;\mathbb{R}^3)$, then $\nabla v\in H^s(\mathbb{R}^3;\mathbb{R}^9)$ with
$$ \|\nabla v\|_{H^s(\mathbb{R}^3)}\le C_s\|\nabla d\|_{H^s(\mathbb{R}^3)}\left(1+\|\nabla d\|_{H^s(\mathbb{R}^3)}\right). $$
Moreover, if $\nabla d\in H^{s+1}(\mathbb{R}^3;\mathbb{R}^3)$, then $\nabla v\in H^{s+1}(\mathbb{R}^3;\mathbb{R}^{9})$ with
$$ \|\nabla^2v\|_{H^s(\mathbb{R}^3)}\le C_s \|\nabla^2d\|_{H^s(\mathbb{R}^3)}\left(1+\|\nabla d\|_{H^s(\mathbb{R}^3)}^2\right). $$
\end{lem}
\begin{proof}\hfill\\
It is easy to see that
$$ \nabla v=\nabla d \left(\cos d \eta+\sin d \omega\right), $$
then using Lemma \ref{l.Hs.es.1}
$$ \|\nabla v\|_{H^s(\mathbb{R}^3)}\lesssim \|\nabla d\|_{H^s(\mathbb{R}^3)}+\|\nabla d\|_{L^3(\mathbb{R}^3)}\|\nabla d\|_{H^s(\mathbb{R}^3)}. $$
Then, since $L^3\hookrightarrow H^s$ for $s>\frac{1}{2}$, we get the first estimate.

\vspace{2mm}

For what concerns the estimate of $\nabla^2 v$, we notice that
$$ \nabla^2 v =\nabla^2d\left(-\sin d \eta+\cos d \omega\right)-[\nabla d]^2\left(\cos d \eta +\sin d \omega\right). $$
By Lemma \ref{l.Hs.es.1}
$$ \|\nabla^2v\|_{H^s(\mathbb{R}^3)}\lesssim \left(\|\nabla^2 d\|_{H^s(\mathbb{R}^3)}+\|\nabla d \nabla d\|_{H^s(\mathbb{R}^3)}\right)\left(1+\|\nabla d\|_{H^s(\mathbb{R}^3)}\right), $$
where we have used again that $L^3\hookrightarrow H^s$. Now it is sufficient to notice that, by the fractional Leibniz rule \eqref{fr.L.}
$$ \|\nabla d\nabla d\|_{H^s(\mathbb{R}^3)}\lesssim \|\nabla d\|_{H^s_6(\mathbb{R}^3)}\|\nabla d\|_{L^3(\mathbb{R}^3)}\lesssim \|\nabla^2 d\|_{H^s(\mathbb{R}^3)}\|\nabla d\|_{H^s(\mathbb{R}^3)}. $$
This concludes the proof.
\end{proof}
\begin{prop}\label{p.reg.fr.}
Let $s\in\left(\frac{1}{2},1\right)$, $T\in(0,\infty]$, $\eta,\omega\in \mathcal{S}^2$ with $\eta\perp\omega$, let
$$ v(x)=\cos d(x) \eta+\sin d(x) \omega\quad \text{for a.e.}\:\:x\in\mathbb{R}^3, $$
with $d\in\Theta^s_T$, then $v\in\Theta^s_T$.
\end{prop}
The proof is the same of Proposition \ref{p.reg.int.}. We have already noticed that Lemma \ref{l.Hs.es.2} does not include an estimate from $\|\nabla v\|_{H^s}$ to $\|\nabla d\|_{H^s}$ as in Lemma \ref{l.Hk.eq.} for $s\in\mathbb{N}$. This makes us unable to prove the local existence result: if we take $v_0\colon\mathbb{R}^3\to \mathcal{S}^2$ with $\nabla v_0\in H^s(\mathbb{R}^3;\mathbb{R}^9)$, even if $v_0\in\text{Span}\{\eta,\omega\}$, we are not able to prove that $\nabla d_0\in H^s(\mathbb{R}^3;\mathbb{R}^3)$. This is no longer true for the global existence, thanks to the following lemma:
\begin{lem}\label{l.Hs.IC.}
    Let $s\in\left(\frac{1}{2},1\right)$, let $\eta,\omega\in \mathcal{S}^2$ with $\eta\perp\omega$, let $v_0\colon\mathbb{R}^3\to \mathcal{S}^2$ with
    $$ v_0(x)=\cos d_0(x)\eta+\sin d_0(x)\omega\quad \text{for a.e.}\:\:x\in\mathbb{R}^3, $$
    then we can find $\varepsilon_0>0$ sufficiently small such that, for any $\varepsilon\in[0,\varepsilon_0)$, if
    $$ \|v_0-\eta\|_{L^\infty(\mathbb{R}^3)}+\|\nabla v_0\|_{H^s(\mathbb{R}^3)}\le \varepsilon, $$
    then $\nabla d_0\in H^s(\mathbb{R}^3;\mathbb{R}^3)$ with
    $$ \|d_0\|_{L^\infty(\mathbb{R}^3)}^2+\|\nabla d_0\|_{H^s(\mathbb{R}^3)}\lesssim \varepsilon. $$
\end{lem}
\begin{proof}\hfill\\
Let us start from the $L^\infty$-norm:
$$ |d_0(x)|^2\lesssim |1-\cos d_0(x)|+|\sin d_0(x)|=|v_0(x)-\eta|\le \varepsilon\quad \text{for a.e.}\:\:x\in\mathbb{R}^3. $$


\noindent On the other hand
 \begin{equation}\label{der.id.}
     \nabla v_0=\nabla d_0(-\sin d_0\eta+\cos d_0\omega),
 \end{equation}
 so it can be seen easily that
 $$ |\nabla v_0|=|\nabla d_0|. $$
 Therefore
 $$ \|\nabla d_0\|_{L^q(\mathbb{R}^3)}=\|\nabla v_0\|_{L^q(\mathbb{R}^3)}\quad \forall q\ge 1. $$
 Now we notice that
 $$ \nabla d_0=\nabla v_0\cdot\omega-(1-\cos d_0)\nabla d_0, $$
 where we have multiplied \eqref{der.id.} by $\omega$ using the fact that $\omega\perp\eta$ and $|\omega|=1$. Now thanks to Lemma \ref{l.Hs.es.1}
 $$ \|\nabla d_0\|_{H^s(\mathbb{R}^3)}\le \|\nabla v_0\cdot \omega\|_{H^s(\mathbb{R}^3)}+\|(1-\cos d_0)\nabla d_0\|_{H^s(\mathbb{R}^3)}\lesssim $$
 $$ \lesssim \|\nabla v_0\|_{H^s(\mathbb{R}^3)}+\|1-\cos d_0\|_{L^\infty(\mathbb{R}^3)}\|\nabla d_0\|_{H^s(\mathbb{R}^3)} + \|\nabla d_0\|_{L^{3}(\mathbb{R}^3)}\|\nabla d_0\|_{H^s(\mathbb{R}^3)}. $$
On the other hand
 $$ \|1-\cos d_0\|_{L^\infty(\mathbb{R}^3)}\lesssim \|d_0\|_{L^\infty(\mathbb{R}^3)}^2 $$
 and
 $$ \|\nabla d_0\|_{L^3(\mathbb{R}^3)}=\|\nabla v_0\|_{L^3(\mathbb{R}^3)}\lesssim \|\nabla v_0\|_{H^s(\mathbb{R}^3)}, $$
 thanks to the condition $s\ge \frac{1}{2}$. Finally,
 $$ \|\nabla d_0\|_{H^s(\mathbb{R}^3)}\lesssim \varepsilon + \varepsilon\|\nabla d_0\|_{H^s(\mathbb{R}^3)}, $$
 so for $\varepsilon\ll1$ we conclude.
\end{proof}
Finally, thanks to Lemma \ref{l.Hs.IC.}, the proof of Theorem \ref{t.ex.fr.} is the same of Theorem \ref{t.ex.int.}.

\newpage

\section{Appendix}

In this section, we prove the technical lemma we have used in the previous sections.

\subsection{Proofs of Lemma from Sections 3 and 4}

\begin{lem}
Let $0\le\alpha_1,\alpha_2<1$ and $T>0$, then it holds
$$ \int_0^t (t-\tau)^{-\alpha_1}\tau^{-\alpha_2}d\tau\le C(\alpha_1,\alpha_2) t^{-\max\{\alpha_1,\alpha_2\}}T^{1-\min\{\alpha_1,\alpha_2\}} \quad \forall t\in(0,T). $$
\end{lem}
\begin{proof}\hfill\\
We split the integral in two pieces:
$$ \int_0^t (t-\tau)^{-\alpha_1}\tau^{-\alpha_2}d\tau= \int_0^{t/2} (t-\tau)^{-\alpha_1}\tau^{-\alpha_2}d\tau+\int_{t/2}^t (t-\tau)^{-\alpha_1}\tau^{-\alpha_2}d\tau. $$
For what concerns the first part, we notice that $t-\tau\ge \frac{t}{2}$, so
$$ \int_0^{t/2} (t-\tau)^{-\alpha_1}\tau^{-\alpha_2}d\tau\lesssim t^{-\alpha_1}\int_0^T\tau^{-\alpha_2}d\tau\simeq t^{-\alpha_1}T^{1-\alpha_2}. $$
On the other hand,
$$ \int_{t/2}^t (t-\tau)^{-\alpha_1}\tau^{-\alpha_2}d\tau\lesssim t^{-\alpha_2}\int_0^T\tau^{-\alpha_1}d\tau\simeq t^{-\alpha_2}T^{1-\alpha_1}. $$
So,
$$ \int_0^t (t-\tau)^{-\alpha_1}\tau^{-\alpha_2}d\tau\lesssim t^{-\alpha_1}T^{1-\alpha_2}+t^{-\alpha_2}T^{1-\alpha_1}. $$
\end{proof}

\begin{lem}
Let $0\le\alpha_1,\alpha_2<1$ and $\beta_1,\beta_2> 1$, then it holds
$$ \int_0^t \frac{1}{((t-\tau)^{\alpha_1}+(t-\tau)^{\beta_1})}\frac{1}{(\tau^{\alpha_2}+\tau^{\beta_2})}d\tau\le C(\alpha_1,\alpha_2,\beta_1,\beta_2) (t^{\max\{\alpha_1,\alpha_2\}}+t^{\min\{\beta_1,\beta_2\}})^{-1}\quad \forall t>0. $$
\end{lem}
\begin{proof}\hfill\\
We split the integral in two pieces:
$$ \int_0^t \frac{1}{((t-\tau)^{\alpha_1}+(t-\tau)^{\beta_1})}\frac{1}{(\tau^{\alpha_2}+\tau^{\beta_2})}d\tau= $$
$$ = \int_0^{t/2} \frac{1}{((t-\tau)^{\alpha_1}+(t-\tau)^{\beta_1})}\frac{1}{(\tau^{\alpha_2}+\tau^{\beta_2})}d\tau + \int_{t/2}^t \frac{1}{((t-\tau)^{\alpha_1}+(t-\tau)^{\beta_1})}\frac{1}{(\tau^{\alpha_2}+\tau^{\beta_2})}d\tau. $$
For what concerns the first part, we notice that $t-\tau\ge \frac{t}{2}$, so
$$ \int_0^{t/2} \frac{1}{((t-\tau)^{\alpha_1}+(t-\tau)^{\beta_1})}\frac{1}{(\tau^{\alpha_2}+\tau^{\beta_2})}d\tau\lesssim (t^{\alpha_1}+t^{\beta_1})^{-1}\int_0^\infty(\tau^{\alpha_2}+\tau^{\beta_2})^{-1}d\tau. $$
Now, since $\alpha_2<1$ and $\beta_2>1$, we have that
$$ \int_0^1 (\tau^{\alpha_2}+\tau^{\beta_2})^{-1}d\tau\le \int_0^1\tau^{-\alpha_2}d\tau\simeq 1; $$
$$ \int_1^\infty (\tau^{\alpha_2}+\tau^{\beta_2})^{-1}d\tau\le \int_1^\infty (1+\tau^{\beta_2})^{-1}d\tau\simeq 1. $$
So
$$ \int_0^{t/2} \frac{1}{((t-\tau)^{\alpha_1}+(t-\tau)^{\beta_1})}\frac{1}{(\tau^{\alpha_2}+\tau^{\beta_2})}d\tau\lesssim (t^{\alpha_1}+t^{\beta_1})^{-1}. $$
On the other hand,
$$ \int_{t/2}^t \frac{1}{((t-\tau)^{\alpha_1}+(t-\tau)^{\beta_1})}\frac{1}{(\tau^{\alpha_2}+\tau^{\beta_2})}d\tau\lesssim (t^{\alpha_2}+t^{\beta_2})^{-1}\int_{t/2}^t((t-\tau)^{\alpha_1}+(t-\tau)^{\beta_1})^{-1}d\tau\le $$
$$ \le (t^{\alpha_2}+t^{\beta_2})^{-1}\int_0^\infty (\tau^{\alpha_1}+\tau^{\beta_1})^{-1}d\tau\simeq (t^{\alpha_2}+t^{\beta_2})^{-1}. $$
\end{proof}

\subsection{Proofs of Lemma from Section 5}

\begin{lem}\label{l.reg.a.}
Let $N\ge  3$, $\ell\in\mathbb{N}$ with $\ell\ge 1$, let $s_1,\ldots,s_\ell\in\mathbb{N}$ and $s\in\mathbb{N}$ with $s\ge \frac{N}{2}-1$ such that
$$ s_1+\ldots+s_\ell\le s-(\ell-1), $$
let $v_j\in H^s(\mathbb{R}^N)$ for $j=1,\ldots,\ell$, then
$$ \|\nabla^{s_1}v_1\cdots\nabla^{s_\ell}v_\ell\|_{L^2(\mathbb{R}^N)}\lesssim \prod_{j=1}^\ell\|v_j\|_{H^s(\mathbb{R}^N)}. $$
\end{lem}
\begin{proof}\hfill\\
We proceed by induction on $\ell$: if $\ell=1$, then $s_1\le k$ so
$$ \|\nabla^{s_1}v\|_{L^2(\mathbb{R}^N)}\le \|v\|_{H^s(\mathbb{R}^N)}. $$
Let us take now $\ell\ge 2$. We consider two cases:
\begin{itemize}
    \item Let us suppose there is $j_0=1,\ldots,\ell$ such that $s-s_{j_0}\ge \frac{N}{2}$. In this case for any $p\in[2,\infty)$
    $$ (1-\Delta)^{\widetilde{s}/2}\nabla^{s_{j_0}}v_{j_0}\in H^{s-s_{j_0}}(\mathbb{R}^N)\hookrightarrow L^p(\mathbb{R}^N). $$
    Then
    $$ \left\|\prod_{j=1}^\ell \nabla^{s_j}v_j\right\|_{L^2(\mathbb{R}^N)}\le \|\nabla^{s_{j_0}}v_{j_0}\|_{L^p(\mathbb{R}^N)}\left\|\prod_{j\neq j_0}\nabla^{s_j}v_j\right\|_{L^q(\mathbb{R}^N)} $$
    with
    $$ \frac{1}{p}+\frac{1}{q}=\frac{1}{2}. $$
    Thanks to the previous remark, we have that for any $p\in[2,\infty)$
    $$ \|\nabla^{s_{j_0}}v_{j_0}\|_{L^p(\mathbb{R}^N)}\lesssim \|\nabla^{s_{j_0}}v_{j_0}\|_{H^{s-s_{j_0}}(\mathbb{R}^N)}\le\|v_{j_0}\|_{H^s(\mathbb{R}^N)}. $$
    On the other hand, if we consider $q=\frac{2N}{N-2}>2$, we have by the Sobolev embedding that
    $$ \left\|\prod_{j\neq j_0}\nabla^{s_j}v_j\right\|_{L^q(\mathbb{R}^N)}\lesssim \left\|\prod_{j\neq j_0}\nabla^{\gamma_j}v_j\right\|_{L^2(\mathbb{R}^N)}, $$
    with
    $$ \sum_{j\neq j_0}\gamma_j=\sum_{j\neq j_0}s_j+1\le s+(\ell-1)+1-s_{j_0}=s-(\ell-2+s_{j_0})\le s-(\ell-1). $$
    So, by inductive hypothesis, we have that
    $$ \left\|\prod_{j\neq j_0}\nabla^{\gamma_j}v_j\right\|_{L^2(\mathbb{R}^N)}\lesssim \prod_{j=1}^{\ell-1}\|v_j\|_{H^s(\mathbb{R}^N)}, $$
    which concludes the thesis in this case.
    \item Let us suppose now $s-s_j<\frac{N}{2}$ for any $j=1,\ldots,\ell$. In this case
    $$ \|\nabla^{s_1}v_1\cdots\nabla^{s_\ell}v_\ell\|_{L^2(\mathbb{R}^N)}\le \prod_{j=1}^\ell \|\nabla^{s_j}v_j\|_{L^{p_j}(\mathbb{R}^N)} $$
    for
    $$ \frac{1}{p_1}+\cdots+\frac{1}{p_\ell}=\frac{1}{2}. $$
    By Sobolev embeddings, we have that
    $$ \nabla^{s_j}v_j\in H^{s-s_j}(\mathbb{R}^N)\hookrightarrow L^p(\mathbb{R}^N)\quad \forall p\in\left[2,\frac{2N}{N-2(s-s_j)}\right],\quad j=1,\ldots,\ell. $$
    So, if we take $p_j=\frac{2N}{N-2(s-s_j)}$ for $j=1,\ldots,\ell$, it is sufficient to see that
    $$ \frac{1}{p_1}+\cdots+\frac{1}{p_\ell}\le \frac{1}{2}. $$
    Precisely
    $$ \frac{1}{p_1}+\cdots+\frac{1}{p_\ell}\le \frac{1}{2}\:\Leftrightarrow\: \sum_{j=1}^{\ell}[N-2(s-s_j)]\le N. $$
    By definition
    $$ \sum_{j=1}^\ell s_j\le s-(\ell-1), $$
    so
    $$ \sum_{j=1}^{\ell}[N-2(s-s_j)]\le \ell(N-2s)+2[s-(\ell-1)]=N+(\ell-1)[N-2s-2] $$
    On the other hand, by hypothesis we have that $s\ge \frac{N}{2}-1$ and $\ell\ge 1$, so we conclude.
\end{itemize}
\end{proof}

\begin{lem}
Let $N\ge  3$, $\ell\in\mathbb{N}$ with $\ell\ge 2$, let $s_1,\ldots,s_\ell\in\mathbb{N}$ and $s\in\mathbb{N}$ with $s\ge \frac{N}{2}-1$ such that
$$ s_1+\ldots+s_\ell\le s-(\ell-2), $$
let $w\in H^{s+1}(\mathbb{R}^N)$, then
$$ \left\|\prod_{j=1}^\ell \nabla^{s_j}w\right\|_{L^2(\mathbb{R}^N)}\lesssim \|\nabla w\|_{H^s(\mathbb{R}^N)}\|w\|_{H^s(\mathbb{R}^N)}^{\ell-1} $$
\end{lem}
\begin{proof}\hfill\\
If there is $s_{j_0}\ge 1$, then
$$ \prod_{j=1}^\ell \nabla^{s_j}w=\prod_{j=1}^\ell \nabla^{\gamma_j}z_j $$
with
$$ \gamma_j=\left\{\begin{array}{ll}
    s_j & j\neq j_0 \\
    s_{j_0}-1 & j=j_0
\end{array}\right. \quad z_j=\left\{\begin{array}{ll}
    w & j\neq j_0 \\
    \nabla w & j=j_0.
\end{array}\right.  $$
In this case we have that
$$ \gamma_1+\ldots \gamma_\ell=s_1+\ldots+s_\ell-1=s-(\ell-1). $$
So, we can apply directly Lemma \ref{l.reg.a.} and we get
$$ \left\|\prod_{j=1}^\ell \nabla^{\gamma_j}z_j\right\|_{L^2(\mathbb{R}^N)}\lesssim \prod_{j=1}^\ell \|z_j\|_{H^s(\mathbb{R}^N)}=\|\nabla w\|_{H^s(\mathbb{R}^N)}\|w\|_{H^s(\mathbb{R}^N)}^{\ell-1}. $$
Let us suppose now $s_j=0$ for any $j=1,\ldots, \ell$. In this case $s=\ell-2$. When $\ell=2$, By Sobolev embedding we have
$$ \|w^2\|_{L^2(\mathbb{R}^N)}\le\|w\|_{L^N(\mathbb{R}^N)}\|w\|_{L^\frac{2N}{N-2}(\mathbb{R}^N)}\lesssim \|w\|_{H^s(\mathbb{R}^N)}\|\nabla w\|_{L^2(\mathbb{R}^N)}. $$
Let us suppose now $\ell>2$:
$$ \|w^\ell\|_{L^2(\mathbb{R}^N)}\le \|w\|_{L^p(\mathbb{R}^N)}\|w^{\ell-1}\|_{L^q(\mathbb{R}^N)} $$
with
$$ \frac{1}{p}+\frac{1}{q}=\frac{1}{2}. $$
If we take $p=N$ and $q=\frac{2N}{N-2}$, by Sobolev embedding we get
$$ \|w\|_{L^p(\mathbb{R}^N)}\|w^{\ell-1}\|_{L^q(\mathbb{R}^N)}\lesssim \|w\|_{H^s(\mathbb{R}^N)}\|\nabla w w^{\ell-2}\|_{L^2(\mathbb{R}^N)}.  $$
Now we notice that
$$ \nabla w w^{\ell-2} = \prod_{j=1}^{\ell-1} z_j, $$
with
$$  z_j=\left\{\begin{array}{ll}
    w & j<\ell \\
    \nabla w & j=\ell.
\end{array}\right.   $$
So we can apply Lemma \ref{l.reg.a.} so that
$$ \|w^\ell\|_{L^2(\mathbb{R}^N)}\lesssim \|\nabla w\|_{H^s(\mathbb{R}^N)}\|w\|_{H^s(\mathbb{R}^N)}^\ell. $$
\end{proof}


\newpage

\clearpage
\bibliographystyle{plain}
 \bibliography{EL_BG}

\begin{thebibliography}{10}

\bibitem{DQ12}
Mimi Dai, Jie Qing, and Maria Schonbek.
\newblock Asymptotic behavior of solutions to liquid crystal systems in
  {$\mathbb{R}^3$}.
\newblock {\em Comm. Partial Differential Equations}, 37(12):2138--2164, 2012.

\bibitem{DH22}
Hengrong Du, Tao Huang, and Changyou Wang.
\newblock Weak compactness property of simplified nematic liquid crystal flows
  in dimension two.
\newblock {\em Math. Z.}, 302(4):2111--2130, 2022.

\bibitem{E61}
Jerald~LaVerne Ericksen.
\newblock Conservation laws for liquid crystals.
\newblock {\em Trans. Soc. Rheol.}, 5:23--34, 1961.

\bibitem{FT09}
Jishan Fan and Tohru Ozawa.
\newblock Regularity criteria for a simplified {E}ricksen-{L}eslie system
  modeling the flow of liquid crystals.
\newblock {\em Discrete Contin. Dyn. Syst.}, 25(3):859--867, 2009.

\bibitem{FP22}
Kamil Fedorowicz and Robert Prosser.
\newblock On the simulation of nematic liquid crystalline flows in a 4:1 planar
  contraction using the {L}eslie-{E}ricksen and {B}eris-{E}dwards models.
\newblock {\em J. Non-Newton. Fluid Mech.}, 310:Paper No. 104949, 17, 2022.

\bibitem{G11}
G.~P. Galdi.
\newblock {\em An introduction to the mathematical theory of the
  {N}avier-{S}tokes equations}.
\newblock Springer Monographs in Mathematics. Springer, New York, second
  edition, 2011.
\newblock Steady-state problems.

\bibitem{GN18}
Yoshikazu Giga and Antonín Novotn\'y.
\newblock {\em Handbook of Mathematical Analysis in Mechanics of Viscous
  Fluids}.
\newblock Springer Cham, 1 edition, 2018.

\bibitem{GK96}
Archil Gulisashvili and Mark~A. Kon.
\newblock Exact smoothing properties of {S}chr\"{o}dinger semigroups.
\newblock {\em Amer. J. Math.}, 118(6):1215--1248, 1996.

\bibitem{HRS20}
Matthias Hieber, James~C. Robinson, and Yoshihiro Shibata.
\newblock {\em Mathematical analysis of the {N}avier-{S}tokes equations},
  volume 2254 of {\em Lecture Notes in Mathematics}.
\newblock Springer, Cham; Fondazione C.I.M.E., Florence, [2020] \copyright
  2020.
\newblock Cetraro, Italy 2017, Selected lectures from the CIME school on
  Mathematical Analysis of the Navier-Stokes Equations: Foundations and
  Overview of Basic Open Problems held September 4--8, 2017, Edited by Giovanni
  P. Galdi and Shibata, Fondazione CIME/CIME Foundation Subseries.

\bibitem{L79}
Frank~Matthews Leslie.
\newblock Theory of flow phenomena in liquid crystals.
\newblock {\em Advances in Liquid Crystals}, 4:1 -- 81, 1979.

\bibitem{LL00}
Z.~Li and B.~Liu.
\newblock On threshold solutions of equivariant
  {C}hern-{S}imons-{S}chr\"{o}dinger equation, 2020.
\newblock Preprint arXiv:2010.09045.

\bibitem{L89}
Fang-Hua Lin.
\newblock Nonlinear theory of defects in nematic liquid crystals; phase
  transition and flow phenomena.
\newblock {\em Comm. Pure Appl. Math.}, 42(6):789--814, 1989.

\bibitem{LL96}
Fang-Hua Lin and Chun Liu.
\newblock Partial regularity of the dynamic system modeling the flow of liquid
  crystals.
\newblock {\em Discrete Contin. Dynam. Systems}, 2(1):1--22, 1996.

\bibitem{LX15}
Shengquan Liu and Xinying Xu.
\newblock Global existence and temporal decay for the nematic liquid crystal
  flows.
\newblock {\em J. Math. Anal. Appl.}, 426(1):228--246, 2015.

\bibitem{OS12}
Christoph Ortner and Endre Suli.
\newblock A note on linear elliptic systems on $\mathbb{R}^d$, 2012.

\bibitem{PZ11}
Marius Paicu and Arghir Zarnescu.
\newblock Global existence and regularity for the full coupled
  {N}avier-{S}tokes and {$Q$}-tensor system.
\newblock {\em SIAM J. Math. Anal.}, 43(5):2009--2049, 2011.

\bibitem{PZ12}
Marius Paicu and Arghir Zarnescu.
\newblock Energy dissipation and regularity for a coupled {N}avier-{S}tokes and
  {$Q$}-tensor system.
\newblock {\em Arch. Ration. Mech. Anal.}, 203(1):45--67, 2012.

\bibitem{W11}
Changyou Wang.
\newblock Well-posedness for the heat flow of harmonic maps and the liquid
  crystal flow with rough initial data.
\newblock {\em Arch. Ration. Mech. Anal.}, 200(1):1--19, 2011.

\bibitem{XZ17}
Fuyi Xu, Xinguang Zhang, Yonghong Wu, and Lishan Liu.
\newblock Global existence and the optimal decay rates for the three
  dimensional compressible nematic liquid crystal flow.
\newblock {\em Acta Appl. Math.}, 150:67--80, 2017.

\end{thebibliography}

\end{document}